\documentclass{article}
\usepackage{graphicx} 
\usepackage{amsthm, amsmath, amssymb, tikz, bm}
\usepackage{mathtools}
\usepackage{xifthen}
\usepackage{comment}
\usepackage[margin=4.5cm]{geometry}
\usepackage[shortlabels]{enumitem}
\usepackage{todonotes}
\usepackage[colorlinks=true,
linkcolor=blue,citecolor=blue,
urlcolor=blue]{hyperref}

\usetikzlibrary{calc,shapes, backgrounds}
\usetikzlibrary[patterns]

\newtheorem{theorem}{Theorem}[section]
\newtheorem{lemma}[theorem]{Lemma}
\newtheorem{sublemma}{}[theorem]
\newtheorem{corollary}[theorem]{Corollary}

\newtheorem{proposition}[theorem]{Proposition}

\newtheorem{Problem}{Problem}

\newcommand{\romannum}[1]{\romannumeral#1\relax}

\makeatletter
\newcommand*{\rom}[1]{\expandafter\@slowromancap\romannumeral #1@}
\makeatother

\title{Unavoidable cycle-contraction minors of large $2$-connected graphs}
\author{Wayne Ge \and James Oxley}

\date{\today}

\begin{document}

\maketitle
\begin{abstract}
    It is well known that every sufficiently large connected graph has, as an induced subgraph, $K_n$, $K_{1,n}$, or an $n$-vertex path. A 2023 paper of Allred, Ding, and Oporowski identified a set of unavoidable induced subgraphs of sufficiently large $2$-connected graphs. In this paper, we establish a dual version of this theorem by focusing on the minors obtained by contracting cycles.
\end{abstract}

\section{Introduction}

For graph and matroid terminology not explicitly defined here, we follow Bondy and Murty~\cite{Bondy-Murty} and Oxley~\cite{Oxl11}. In particular, we allow graphs to have loops and parallel edges; a graph is {\it simple} if it has neither. A {\it bond} in a graph is a minimal edge cut.

For an integer $k$ exceeding one, a graph $G$ is {\it $k$-connected} if $|V(G)|\geq k$ and, whenever $u$ and $v$ are distinct vertices of $G$, there are at least $k$ pairwise internally disjoint $uv$-paths. In particular, for $n\geq2$, the {\it bond graph $B_n$} that consists of two vertices joined by $n$ parallel edges is $2$-connected. Observe that our definition of a $k$-connected graph is broader than that of many authors who require that $|V(G)|\geq k+1$ for such a graph $G$. A graph $H$ with at least two vertices is {\it $k$-edge-connected} if $H\setminus Z$ is connected for all subsets $Z$ of $E(H)$ with $|Z|<k$. We note that a single-vertex graph is neither $2$-connected nor $2$-edge-connected.

\subsection{Unavoidable families}

The following theorem of Ramsey~\cite{Ramsey1930} has led to the development of a widely studied area of combinatorics.

\begin{theorem}\label{thm:ramsey}
    Let $r$ be a positive integer. There is an integer $R(r)$ such that every graph on at least $R(r)$ vertices has, as an induced subgraph, either $K_r$ or its complement $\overline{K_r}$.
\end{theorem}

Ramsey's Theorem identifies a family of graphs, namely, $\{K_r,\overline{K_r}\}$, such that every sufficiently large graph must have a member of this family as an induced subgraph. There are many other results of this type. For example, a well-known theorem (see, for instance,~\cite[Proposition~9.4.1]{Diestel}) states that every sufficiently large connected graph has, as an induced subgraph, $K_n$, $K_{1,n}$, or an $n$-vertex path. Allred, Ding, and Oporowski~\cite{SDO2020} determined a family of $2$-connected graphs such that every sufficiently large $2$-connected graph has a member of this family as an induced subgraph. The aim of this paper is to establish a dual of Allred, Ding, and Oporowski's result. 

To distinguish our setting from other problems in Ramsey theory, we refer to such results as a search for \emph{unavoidable families}. Let $\mathcal{F}$ be a family of graphs. An \emph{unavoidable-families characterization} consists of a subfamily $\mathcal{U}$ of $\mathcal{F}$ and a relation $\leq$ such that, for every sufficiently large graph $G$ in $\mathcal{F}$, there is a large graph $H$ in $\mathcal{U}$ such that  $H \leq G$. To describe what “sufficiently large” means, one typically specifies the existence of a function such as $R(r)$ in Theorem~\ref{thm:ramsey}. While much work has gone into improving the upper and lower bounds for such functions, this paper, like many others in the area, concerns itself only with the existence of such functions.

While the three theorems mentioned above characterize unavoidable induced subgraphs for different connectivity classes, such characterizations for topological minors have been given for $3$- and internally-$4$-connected graphs~\cite{OOT1993}. For parallel minors, such characterizations are given for $1$-, $2$-, $3$-, and internally-$4$-connected graphs~\cite{Chun2009}. In the next section, we introduce cycle-contraction minors of graphs and note how such minors are related through duality to induced subgraphs. The main result of this paper is an unavoidable-families characterization for cycle-contraction minors of 2-connected graphs.

\subsection{cc-minors}
A graph $K$ is a {\it cycle-contraction minor} or {\it cc-minor} of a graph $J$ if there is a sequence $J_0,J_1,\dots,J_n$ of graphs such that $(J_0,J_n)=(J,K)$ and, for each $i\in\{1,2,\dots,n\}$, there is a cycle $C_{i-1}$ of $J_{i-1}$ such that $J_{i}=J_{i-1}/C_{i-1}$. In this paper, we determine a list of loopless $2$-connected graphs such that every sufficiently large $2$-connected graph has a member of the list as a cc-minor. Unless otherwise stated, every cycle contraction we perform is accompanied by the contraction of all of the loops it creates, where the contraction of a loop has the same effect as the deletion of the corresponding edge. Because each loop is itself a cycle, each contraction of a loop is an example of the contraction of a cycle.

Let $G$ be a graph and let $H$ be an induced subgraph of $G$. Clearly, $H$ can be obtained from $G$ by a sequence of operations, each consisting of either deleting a bond from the current graph or deleting an isolated vertex. When $G$ is a plane graph with planar dual $G^*$, the planar dual $H^*$ of $H$ is obtained from $G^*$ by a sequence of operations, each consisting of contracting a cycle in the current graph. Although, in the planar setting, the dual operation of deleting a bond is contracting a cycle, we need more care in specifying the dual of the operation of taking an induced subgraph. This is because a graph obtained from $G$ by repeatedly deleting bonds and isolated vertices need not be an induced subgraph of $G$. For example, in $C_4$, a matching consisting of two edges $e$ and $f$ forms a bond, yet $C_4 \setminus \{e,f\}$ is not an induced subgraph of $C_4$. In Section~\ref{sec:pre}, we shall prove the following result, which links induced subgraphs and cc-minors via duality. 

\begin{lemma}\label{cc} 
Let $G$ be a loopless $2$-connected plane graph. A graph $H$ is a $2$-connected induced subgraph of $G$ if and only if $H^*$ is a $2$-connected cc-minor of $G^*$. 
\end{lemma}

The statements of both our main result and of the theorem of Allred, Ding, and Oporowski~\cite{SDO2020} will rely on Tutte's tree decomposition result for $2$-connected graphs, which we shall introduce next.

Let $G_1$ and $G_2$ be graphs such that $V(G_1)\cap V(G_2)=\{u,v\}$ and $E(G_1)\cap E(G_2)=\{e\}$ where $e$ is neither a loop nor a cut edge of $G_1$ or $G_2$. The graph $G_1\cup G_2$ is the {\it parallel connection} of $G_1$ and $G_2$ with {\it basepoint $e$}. The graph obtained from $G_1\cup G_2$ by deleting $e$ is the {\it $2$-sum}, $G_1\bigoplus_2 G_2$, of $G_1$ and $G_2$ with {\it basepoint $e$}.

A \emph{graph-labeled tree} is a tree $T$ with vertex set $\{G_1, G_2, \dots, G_k\}$ for some positive integer $k$, satisfying the following conditions for all distinct $i,j \in \{1,2,\dots,k\}$:
\begin{itemize}
    \item[(\romannum{1})] $G_i$ is a graph;
    \item[(\romannum{2})] if $G_{i}$ and $G_{j}$ are joined by an edge $e$ of $T$, then $E(G_{i})\cap E(G_{j})=\{e\}$ and $e$ is neither a loop nor a cut edge of $G_{i}$ or $G_{j}$; and
    \item[(\romannum{3})] if $G_{i}$ and $G_{j}$ are non-adjacent, then $E(G_{i})\cap E(G_{j})$ is empty.
\end{itemize}

We call $G_1,G_2,\dots,G_k$ the {\it vertex labels} of $T$; for each $i$ in $\{1,2,\dots,k\}$, the edges in $E(G_{i})\cap E(T)$ are called the {\it basepoints} of $G_i$.

Let $e$ be an edge of a graph-labeled tree $T$ and suppose $e$ joins the vertices $H_1$ and $H_2$. If we contract $e$ from $T$ and relabel by $H_1\bigoplus_2 H_2$ the vertex that results by identifying the endpoints of $e$ leaving all other edge and vertex labels unchanged, then we get a new graph-labeled tree, $T/e$.

A {\it tree decomposition} of a loopless $2$-connected graph $G$ is a graph-labeled tree $T$ such that if $V(T)=\{G_1,G_2,\dots,G_k\}$ and $E(T)=\{e_1,e_2,\dots,e_{k-1}\}$, then
\begin{itemize}
    \item[(\romannum{1})] $E(G)=(E(G_1)\cup E(G_2)\cup\dots\cup E(G_k))-\{e_1,e_2,\dots,e_{k-1}\}$;
    \item[(\romannum{2})] $|E(G_i)|\geq 3$ for all $i$ unless $|E(G)|<3$, in which case, $k=1$ and $G_1=G$; and
    \item[(\romannum{3})] $G$ is the graph that labels the single vertex of $T/e_1,e_2,\dots,e_{k-1}$.
\end{itemize}

Tutte~\cite{Tutte1966} proved that every 2-connected graph has a tree decomposition in which the vertex labels are restricted. However, such a tree decomposition need not be unique (see, for example,~\cite[p.308]{Oxl11}).

\begin{theorem}\label{thm:tree-decomp}
    Let $G$ be a loopless $2$-connected graph. Then $G$ has a tree decomposition in which every vertex label is a simple $3$-connected graph, a copy of $K_3$, or a copy of $B_3$. Moreover, each vertex label is isomorphic to a minor of $G$.
\end{theorem}

Next we use tree decompositions to state Allred, Ding, and Oporowski's~\cite{SDO2020} identification of a set of unavoidable $2$-connected induced subgraphs of large simple 2-connected graphs.

\begin{theorem}\label{thm:unavoidable induced subgraph}
    Let $k$ be an integer exceeding two. Then there is an integer $f(k)$ such that every simple $2$-connected graph with at least $f(k)$ vertices has, as an induced subgraph, one of
    \begin{itemize}
        \item[(\romannum{1})]  $K_k$;
        \item[(\romannum{2})]  a subdivision of $K_{2,k}$;
        \item[(\romannum{3})]  a graph that is obtained from a subdivision of $K_{2,k}$ by adding an edge joining the two degree-$k$ vertices; or
        \item[(\romannum{4})]  a $k$-vertex graph having a tree decomposition whose underlying tree is a path $P$, such that each vertex of $P$ is labeled by
        \begin{itemize}
            \item[(a)] a copy of $K_4$ in which the basepoints form a matching, or
            \item[(b)] a copy of $K_3$ or $B_3$,
        \end{itemize}
        where neither end of $P$ is labeled by $B_3$, and no two consecutive vertices of $P$ are labeled by $B_3$.

    \end{itemize}
\end{theorem}

The following theorem is the main result of this paper. A graph $H$ is a \textit{parallel extension} of a graph $G$ if $H$ can be obtained from $G$ by, for each non-loop edge $e$ in $G$, deleting $e$ and joining the ends of $e$ by a non-empty set of parallel edges. Similarly, a graph $J$ is a \textit{parallel-path extension} of a graph $K$ if $J$ can be obtained from $K$ by, for each non-loop edge $e$ in $K$, deleting $e$ and adding a non-empty set of internally disjoint paths, each of which contains at least one edge and connects the ends of $e$.

\begin{theorem}\label{main}
    Let $r$ be a positive integer. There is an integer $g(r)$ such that every loopless $2$-connected graph $G$ with $|E(G)| \geq g(r)$ has, as a cc-minor, a parallel-path extension of a graph having a tree decomposition whose underlying tree is a path on at least $r$ vertices, where each vertex is labeled either by a copy of $K_3$ or $B_3$, or by a copy of $K_4$ in which the basepoints form a matching.
\end{theorem}

We note that, since cycle contraction does not preserve simplicity, we do not require a cc-minor to be simple. Consequently, for certain large graphs, such as large cliques, the unavoidable cc-minor in our list is a bond graph consisting of many parallel edges. As will be explained in Section~\ref{sec:pre}, a bond graph admits a tree decomposition whose underlying tree is a path, with each vertex labeled by $B_3$.

\subsection{A matroid perspective}

While the current paper focuses mainly on cc-minors of graphs, this notion can be naturally generalized to matroids. A reader who is only interested in graphs may ignore this subsection. A reader familiar with matroids will find that this subsection provides additional motivation for the material in this paper. While a comprehensive treatment of matroids may be found in~\cite{Oxl11}, we refer to~\cite{newOxl} for a concise overview of the definitions needed to understand the following section.

A matroid $N$ is a \textit{circuit-contraction minor} or \textit{cc-minor} of a matroid $M$ if there is a sequence $M_0, M_1, \dots, M_n$ of matroids such that $(M_0, M_n) = (M, N)$ and, for each $i \in \{1, 2, \dots, n\}$, there is a circuit $C_{i-1}$ of $M_{i-1}$ with $M_i = M_{i-1} / C_{i-1}$.

On the other hand, an \textit{induced restriction} of a matroid $M$ is a matroid $N$ obtained from $M$ by restricting to a flat. Not only do induced restrictions and circuit-contraction minors extend the notions of induced subgraphs and cycle-contraction minors to matroids, but the graph duality described in Lemma~\ref{cc} is, in fact, a special case of the following.  

\begin{proposition}
    A matroid $N$ is an induced restriction of a matroid $M$ if and only if $N^*$ is a cc-minor of $M^*$.
\end{proposition}

\begin{proof}
    Suppose $N$ is an induced restriction of $M$. Then there is a sequence $M_0, M_1, \dots, M_n$ of matroids such that $(M_0, M_n) = (M, N)$ and, for each $i \in \{1, 2, \dots, n\}$, the matroid $M_i$ is obtained from $M_{i-1}$ by restricting to a hyperplane $H_{i-1}$ of $M_{i-1}$. Taking duals, we obtain a sequence $M_0^*, M_1^*, \dots, M_n^*$ of matroids such that $(M_0^*, M_n^*) = (M^*, N^*)$ and, for each $i \in \{1, 2, \dots, n\}$, the matroid $M_i^*$ is obtained from $M_{i-1}^*$ by contracting the circuit $E(M_{i-1})-H_{i-1}$. Thus, $N^*$ is a cc-minor of $M^*$.  

    The converse follows by reversing the argument just given.
\end{proof}

One of our motivations is to extend Theorem~\ref{thm:unavoidable induced subgraph} to regular matroids by addressing the following.

\begin{Problem}\label{question:regular matroid}
     Find a set of unavoidable induced restrictions for sufficiently large $2$-connected simple regular matroids.
\end{Problem}

Seymour~\cite{Seymour1980} showed that every regular matroid can be constructed from graphic matroids, cographic matroids, and copies of the $10$-element matroid $R_{10}$. The allowable operations in this construction are direct sums, $2$-sums, and $3$-sums. While Theorem~\ref{thm:unavoidable induced subgraph} addresses the case in which a regular matroid has a large graphic matroid as an induced restriction, Theorem~\ref{main} addresses the case in which a large cographic matroid appears. Thus, Theorem~\ref{main} is not only the dual of Theorem~\ref{thm:unavoidable induced subgraph}, but also a crucial step toward solving Problem~\ref{question:regular matroid}.

Since Theorem~\ref{main} answers Problem~\ref{question:regular matroid} for cographic matroids by solving the dual problem for graphic matroids, we observe that the requirement that $M^*(G)$ be simple is equivalent to the requirement that the matroid $M^*(G)$ has no circuit of size less than three. Equivalently, the graph $G$ has no bond of size less than three, that is, $G$ is $3$-edge-connected. It is this equivalence that explains why Theorem~\ref{main} does not confine attention to simple graphs.

\section{Preliminaries}\label{sec:pre}
This section presents some definitions along with a proof of Lemma~\ref{cc}. Specifically, we define a class of graphs called $r$-templates, that have tree decompositions exhibiting a particular structure. Additionally, we motivate the study of cc-minors, showing how such minors relate to induced subgraphs and minors.

\subsection{$r$-templates}
Figure~\ref{fig:Sample decomp} provides an example of a 2-connected graph $G$ with a tree decomposition, as described in Theorem~\ref{thm:tree-decomp}. Although, for an arbitrary 2-connected graph $G$, such a graph-labeled tree may exhibit an arbitrary structure, we introduce a class of 2-connected graphs that have tree decompositions with specific, notable features.

\begin{figure}[htb]
\hbox to \hsize{
\hfil
\resizebox{8.5cm}{!}{\input{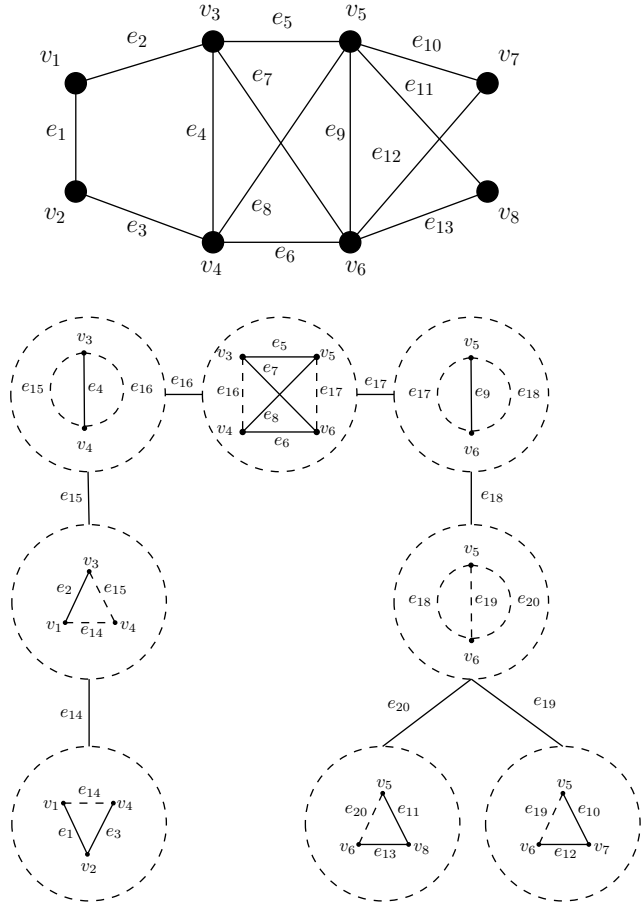}}%
\hfil
}
\caption{A $2$-connected graph $G$ and its tree decomposition}\label{fig:Sample decomp}
\end{figure}

A {\it fan graph} $F_n$ is a simple graph that is obtained from an $n$-vertex path $v_1v_2\dots v_n$ by joining each $v_i$ to a new vertex $u$. We call each $uv_i$ edge a {\it spoke}. In particular, each edge joining $u$ with $v_1$ or $v_n$ is an {\it outer spoke}. For positive integers $t_1,t_2,\dots, t_n$, we obtain a {\it fan-type  graph} $F_{t_1,t_2,\dots,t_n}$ by replacing each spoke $uv_i$ of $F_n$ by $t_i$ parallel edges. In Figure~\ref{fan examples}, we show three different examples of fan-type graphs. Note that bond graphs are considered to be fan-type graphs, since they arise from single-vertex paths.
\begin{center}
    \begin{figure}[htb]
    \hbox to \hsize{
	\hfil
	\resizebox{11cm}{!}{\input{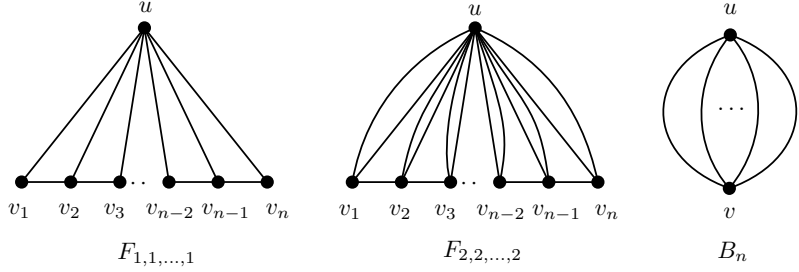}}%
	\hfil
    }
    \caption{Three examples of fan-type graphs}\label{fan examples}
\end{figure}
\end{center}

An \textit{$r$-template} is a $2$-connected graph $G$ that can be obtained from an $r$-vertex path $P_r$ by using the following operations.
\begin{itemize}
    \item[(\romannum{1})] Label each vertex of $P_r$ by $K_4$, $K_3$, or $B_3$. Such vertex labels are called {\it parts}.
    \item[(\romannum{2})] For each part, pick distinct basepoints, one for each of its adjacent parts. Moreover, if a part that is labeled by $K_4$ is adjacent to two other parts, then we always pick two non-adjacent edges as basepoints in that $K_4$.
    \item[(\romannum{3})] Apply $2$-sums across the specified basepoints.
\end{itemize}

Figure~\ref{Building block} shows the possible parts for templates, and Figure~\ref{template} shows the construction process of a sample $6$-template. The following are some special examples of templates.
\begin{figure}[htb]
    \hbox to \hsize{
	\hfil
	\resizebox{11cm}{!}{\tikzset{every picture/.style={line width=0.75pt}} 

\begin{tikzpicture}[x=0.75pt,y=0.75pt,yscale=-1,xscale=1]

\draw  [color={rgb, 255:red, 0; green, 0; blue, 0 }  ,draw opacity=1 ][dash pattern={on 4.5pt off 4.5pt}] (490.35,128.15) .. controls (490.35,89.63) and (521.58,58.4) .. (560.1,58.4) .. controls (598.63,58.4) and (629.86,89.63) .. (629.86,128.15) .. controls (629.86,166.68) and (598.63,197.9) .. (560.1,197.9) .. controls (521.58,197.9) and (490.35,166.68) .. (490.35,128.15) -- cycle ;
\draw [color={rgb, 255:red, 0; green, 0; blue, 0 }  ,draw opacity=1 ][dash pattern={on 4.5pt off 4.5pt}] (275.18,128.15) .. controls (275.18,89.63) and (306.41,58.4) .. (344.93,58.4) .. controls (383.46,58.4) and (414.68,89.63) .. (414.68,128.15) .. controls (414.68,166.68) and (383.46,197.9) .. (344.93,197.9) .. controls (306.41,197.9) and (275.18,166.68) .. (275.18,128.15) -- cycle ;
\draw  [color={rgb, 255:red, 0; green, 0; blue, 0 }  ,draw opacity=1 ][dash pattern={on 4.5pt off 4.5pt}] (60.02,128.15) .. controls (60.02,89.63) and (91.24,58.4) .. (129.77,58.4) .. controls (168.29,58.4) and (199.52,89.63) .. (199.52,128.15) .. controls (199.52,166.68) and (168.29,197.9) .. (129.77,197.9) .. controls (91.24,197.9) and (60.02,166.68) .. (60.02,128.15) -- cycle ;
\draw [color={rgb, 255:red, 0; green, 0; blue, 0 }  ,draw opacity=1 ] [dash pattern={on 4.5pt off 4.5pt}]  (93.7,91.58) -- (93.7,164.73) ;
\draw [color={rgb, 255:red, 0; green, 0; blue, 0 }  ,draw opacity=1 ] [dash pattern={on 4.5pt off 4.5pt}]  (165.83,91.58) -- (165.83,164.73) ;
\draw  [fill={rgb, 255:red, 0; green, 0; blue, 0 }  ,fill opacity=1 ] (91.24,91.58) .. controls (91.24,90.22) and (92.34,89.11) .. (93.7,89.11) .. controls (95.06,89.11) and (96.16,90.22) .. (96.16,91.58) .. controls (96.16,92.94) and (95.06,94.04) .. (93.7,94.04) .. controls (92.34,94.04) and (91.24,92.94) .. (91.24,91.58) -- cycle ;
\draw  [fill={rgb, 255:red, 0; green, 0; blue, 0 }  ,fill opacity=1 ] (91.24,164.73) .. controls (91.24,163.37) and (92.34,162.27) .. (93.7,162.27) .. controls (95.06,162.27) and (96.16,163.37) .. (96.16,164.73) .. controls (96.16,166.09) and (95.06,167.19) .. (93.7,167.19) .. controls (92.34,167.19) and (91.24,166.09) .. (91.24,164.73) -- cycle ;
\draw  [fill={rgb, 255:red, 0; green, 0; blue, 0 }  ,fill opacity=1 ] (163.37,91.58) .. controls (163.37,90.22) and (164.47,89.11) .. (165.83,89.11) .. controls (167.19,89.11) and (168.29,90.22) .. (168.29,91.58) .. controls (168.29,92.94) and (167.19,94.04) .. (165.83,94.04) .. controls (164.47,94.04) and (163.37,92.94) .. (163.37,91.58) -- cycle ;
\draw  [fill={rgb, 255:red, 0; green, 0; blue, 0 }  ,fill opacity=1 ] (163.37,164.73) .. controls (163.37,163.37) and (164.47,162.27) .. (165.83,162.27) .. controls (167.19,162.27) and (168.29,163.37) .. (168.29,164.73) .. controls (168.29,166.09) and (167.19,167.19) .. (165.83,167.19) .. controls (164.47,167.19) and (163.37,166.09) .. (163.37,164.73) -- cycle ;
\draw    (93.7,91.58) -- (165.83,91.58) ;
\draw    (93.7,164.73) -- (165.83,164.73) ;
\draw    (93.7,91.58) -- (165.83,164.73) ;
\draw    (93.7,164.73) -- (165.83,91.58) ;
\draw [color={rgb, 255:red, 0; green, 0; blue, 0 }  ,draw opacity=1 ] [dash pattern={on 4.5pt off 4.5pt}]  (344.35,88.49) -- (309.64,162.41) ;
\draw [color={rgb, 255:red, 0; green, 0; blue, 0 }  ,draw opacity=1 ] [dash pattern={on 4.5pt off 4.5pt}]  (344.35,88.49) -- (381.77,162.41) ;
\draw  [fill={rgb, 255:red, 0; green, 0; blue, 0 }  ,fill opacity=1 ] (341.89,88.49) .. controls (341.89,87.13) and (342.99,86.03) .. (344.35,86.03) .. controls (345.71,86.03) and (346.81,87.13) .. (346.81,88.49) .. controls (346.81,89.85) and (345.71,90.95) .. (344.35,90.95) .. controls (342.99,90.95) and (341.89,89.85) .. (341.89,88.49) -- cycle ;
\draw  [fill={rgb, 255:red, 0; green, 0; blue, 0 }  ,fill opacity=1 ] (307.18,162.41) .. controls (307.18,161.05) and (308.28,159.95) .. (309.64,159.95) .. controls (311,159.95) and (312.1,161.05) .. (312.1,162.41) .. controls (312.1,163.77) and (311,164.88) .. (309.64,164.88) .. controls (308.28,164.88) and (307.18,163.77) .. (307.18,162.41) -- cycle ;
\draw  [fill={rgb, 255:red, 0; green, 0; blue, 0 }  ,fill opacity=1 ] (379.31,162.41) .. controls (379.31,161.05) and (380.41,159.95) .. (381.77,159.95) .. controls (383.13,159.95) and (384.23,161.05) .. (384.23,162.41) .. controls (384.23,163.77) and (383.13,164.88) .. (381.77,164.88) .. controls (380.41,164.88) and (379.31,163.77) .. (379.31,162.41) -- cycle ;
\draw    (309.64,162.41) -- (381.77,162.41) ;
\draw [color={rgb, 255:red, 0; green, 0; blue, 0 }  ,draw opacity=1 ] [dash pattern={on 4.5pt off 4.5pt}]  (560.35,89.23) .. controls (539.99,114.01) and (541.8,146.59) .. (559.86,167.07) ;
\draw [color={rgb, 255:red, 0; green, 0; blue, 0 }  ,draw opacity=1 ] [dash pattern={on 4.5pt off 4.5pt}]  (559.86,167.07) .. controls (580.21,142.29) and (578.4,109.71) .. (560.35,89.23) ;
\draw    (560.35,89.23) -- (559.86,167.07) ;
\draw  [fill={rgb, 255:red, 0; green, 0; blue, 0 }  ,fill opacity=1 ] (557.46,89.23) .. controls (557.46,87.64) and (558.75,86.34) .. (560.35,86.34) .. controls (561.94,86.34) and (563.24,87.64) .. (563.24,89.23) .. controls (563.24,90.83) and (561.94,92.12) .. (560.35,92.12) .. controls (558.75,92.12) and (557.46,90.83) .. (557.46,89.23) -- cycle ;
\draw  [fill={rgb, 255:red, 0; green, 0; blue, 0 }  ,fill opacity=1 ] (556.97,167.07) .. controls (556.97,165.48) and (558.26,164.19) .. (559.86,164.19) .. controls (561.46,164.19) and (562.75,165.48) .. (562.75,167.07) .. controls (562.75,168.67) and (561.46,169.96) .. (559.86,169.96) .. controls (558.26,169.96) and (556.97,168.67) .. (556.97,167.07) -- cycle ;

\end{tikzpicture}}%
	\hfil
    }
    \caption{Possible parts of templates}\label{Building block}
\end{figure}
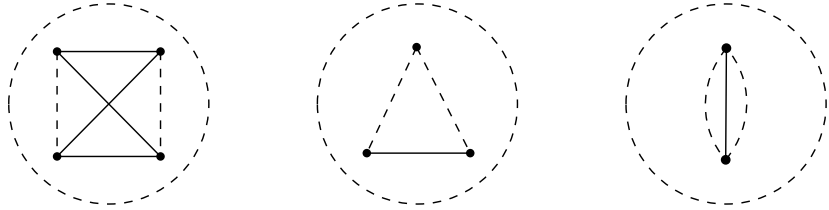

\begin{figure}[htb]
    \hbox to \hsize{
	\hfil
	\resizebox{11.5cm}{!}{\input{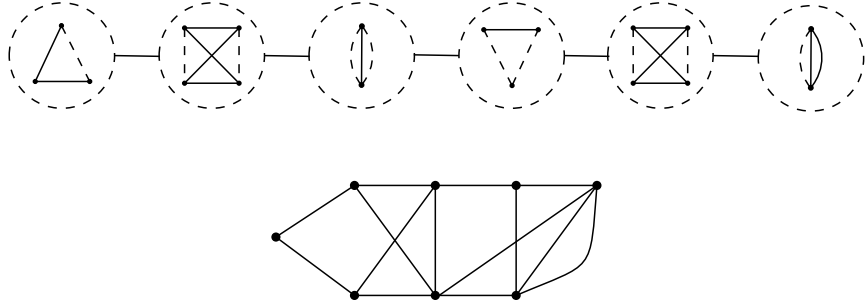}}%
	\hfil
    }
    \caption{A sample $6$-template}\label{template}
\end{figure}

\begin{itemize}
    \item[(\romannum{1})] A fan graph $F_n$ is a $(2n-3)$-template in which the parts alternate between $K_3$ and $B_3$, beginning and ending with $K_3$.
    \item[(\romannum{2})] A fan-type graph $F_{t_1,t_2,\dots,t_n}$ other than $B_2$ is a $((\sum_{i=1}^nt_i)+n-3)$-template.
    \item[(\romannum{3})] For $n\geq 3$, a bond graph $B_n$ is an $(n-2)$-template for which every part is a $B_3$.
    \item[(\romannum{4})] For $n\geq 3$, a cycle $C_n$ is an $(n-2)$-template  for which every part is a $K_3$.
\end{itemize}

The main result of this paper asserts that a set of unavoidable cc-minors of a sufficiently large 2-connected, loopless graph consists of parallel-path extensions of large templates. Before presenting the proofs, we provide motivation for studying cc-minors in the context of graph theory.

\subsection{cc-minors of graphs}

Graph relations, including induced subgraphs and minors, have garnered significant attention in various research areas. However, the study of cc-minors in graphs remains relatively unexplored. Although this relation is novel in certain respects, it has a strong connection to other graph relations. Here, we demonstrate how cc-minors relate to induced subgraphs through duality.\\

\noindent{\it Proof of Lemma~\ref{cc}.} First observe that, the graph $H$ is $2$-connected if and only if its dual $H^*$ is $2$-connected.

Now assume that $H$ is a 2-connected induced subgraph of $G$. Then $H$ can be obtained from $G$ by consecutively deleting some vertices $v_1,v_2,\dots,v_m$. Let $G_0=G$ and, for each $i\in\{1,2,\dots,m\}$, let $G_i=G_{i-1}-v_i$. Since $H$ is a connected subgraph of $G$, we may choose the deletion order so that, for each $i\in\{1,2,\dots,m\}$, the vertex $v_i$ is not a cut vertex of $G_{i-1}$ and therefore the set of edges of $G_{i-1}$ meeting $v_i$ is a bond $B_i$. As $B_i$ is a cycle in the dual of $G_{i-1}$, we see that $H^*=G_m^*=G^*/B_1/B_2/\dots/B_m$ is a cc-minor of $G^*$.

Conversely, assume that $H^*$ is a $2$-connected cc-minor of $G^*$. Let $Z=E(G^*)-E(H^*)$. Because $H^*$ is obtained from $G^*$ by repeatedly contracting cycles, we have $H^*=G^*/Z$. By the construction of $H^*$, every edge of $Z$ is in a cycle of $G^*$ that is contained in $Z$. Since $H^*$ is 2-connected, for each cycle $C$ in the plane graph $G^*$ such that $C\subseteq Z$, either all of the edges in the interior of $C$ are in $Z$ or all of the edges in the exterior of $C$ are in $Z$, but not both; otherwise $H^*$ would have a cut vertex. In the first case, we color the faces in the interior of $C$ red. In the second case, we color the faces in the exterior of $C$ red. In particular, if $F\subseteq Z$ and $F$ is a cycle bounding a face of $G^*$, then that face is colored red. Now, in $G$, consider the set $R$ of vertices that correspond to the red faces of $G^*$.  Because every edge in $Z$ is contained in a red face, contracting $Z$ in $G^*$ has the same effect as contracting all red faces. Moreover, the dual of contracting a face is deleting the corresponding vertex in the dual. Therefore, deleting the vertices of $R$ from $G$ gives the dual of the graph $H^*$.\qed\\

Rather than defining a graph minor through the local operations of deletion and contraction, it can also be characterized by its global structure.

\begin{proposition}
   A graph $G$ has a graph $H$ as a minor if and only if $G$ has a set $\{ G_v : v \in V(H) \}$ of disjoint connected subgraphs and a set $\{ f_e : e \in E(H) \}$ of distinct edges that is disjoint from $\cup_{v\in V(H)}E(G_v)$ such that, for every edge $e \in E(H)$ having ends $u$ and $v$, the ends of $f_e$ are contained in $G_u$ and $G_v$, respectively. 
\end{proposition}

We can characterize a cc-minor in a similar way. We note that, for the next proposition only, when a cycle contraction generates a loop, we do not require the resulting loop to be contracted.

\begin{proposition}\label{prop:structure}
A loopless graph $G$ contains a graph $H$ as a cc-minor if and only if $G$ has
\begin{enumerate}
    \item[(\romannum{1})] a collection $\{ G_v : v \in V(H) \}$ of disjoint subgraphs such that
    \begin{itemize}
        \item[(a)] each $G_v$ is either a $2$-edge-connected subgraph of $G$ or a single-vertex subgraph of $G$, and
        \item[(b)] $\bigcup_{v \in V(H)} V(G_v) = V(G)$;
    \end{itemize}
    and
    \item[(\romannum{2})] a set $\{ f_e : e \in E(H) \}$ of distinct edges in $G$ that is the complement of $\bigcup_{v \in V(H)} E(G_v)$ such that, for every edge $e \in E(H)$ with endpoints $u$ and $v$, the endpoints of $f_e$ lie in $G_u$ and $G_v$, respectively.
\end{enumerate}
\end{proposition}

The proof of Proposition~\ref{prop:structure} is deferred to Section~\ref{2-con}.

\section{Internally disjoint $XY$-paths}\label{lemmas}

Let $G$ be a graph and $X,Y$ be two disjoint set of vertices of $G$. An \textit{$XY$-path} $P$ is an $xy$-path such that there are vertices $x$ and $y$ for which $V(P)\cap X=\{x\}$ and $V(P)\cap Y=\{y\}$. When $H$ and $K$ are disjoint subgraphs of $G$, a $V(H)V(K)$-path will sometimes be called an {\it $HK$-path}. Two $XY$-paths $P_1$ and $P_2$ are {\it internally disjoint} if $(V(P_1)\cap V(P_2))-X-Y=\emptyset$.

\begin{lemma}\label{xy-path}
    Let $G$ be a graph, and let $X$ and $Y$ be two disjoint sets of vertices in $G$. For a cycle $C$ in $G$ such that $V(C) \cap X \neq \emptyset$ and $V(C) \cap Y = \emptyset$, define $X' = V(G[X \cup V(C)]/C)$. If there are $k$ internally disjoint $XY$-paths in $G$, then there are at least $k$ internally disjoint $X'Y$-paths in $G/C$.
\end{lemma}

\begin{proof}
    Let $P_1, P_2, \dots, P_k$ be $k$ internally disjoint $XY$-paths in $G$. Since contracting $C$ can be achieved by repeatedly contracting single edges of $C$, it suffices to establish the following assertion.

    \begin{sublemma}\label{repeatedly_contract}
        If $G' = G/e$ where $e$ joins $u$ and $v$, with $\{u,v\} \cap X \neq \emptyset$ and $\{u,v\} \cap Y = \emptyset$, then $G'$ contains $k$ internally disjoint $X'Y$-paths, where $X' = V(G[X \cup \{u,v\}]/e)$.
    \end{sublemma}

    If $\{u,v\} \subseteq X$ or if $u \in X$ and $v \notin V(P_i)$ for all $i \in \{1,2,\dots,k\}$, then $P_1, P_2, \dots, P_k$ are internally disjoint $X'Y$-paths in $G'$. Now, assume $u \in X$ and $v \in V\left(\bigcup_{i \in \{1,2,\dots,k\}} P_i\right) - X$. Since $P_1, P_2, \dots, P_k$ are internally disjoint and $v \notin Y$, there is exactly one path, say $P_1$, containing $v$. Let $y$ denote the unique vertex in $Y \cap V(P_1)$, and let $P_1'$ be the $yv$-subpath of $P_1$. Then $P_1', P_2, \dots, P_k$ form $k$ internally disjoint $X'Y$-paths in $G'$. This completes the proof of ~\ref{repeatedly_contract}, and the lemma follows.
\end{proof}

\begin{lemma}\label{k-edge-connected}
    For a positive integer $k$, let $G$ be a $k$-edge-connected graph, and let $H$ be a cc-minor of $G$ with at least two vertices. Then $H$ is $k$-edge-connected.
\end{lemma}

\begin{proof}
    It suffices to show that if $e$ is an edge of a graph $G$ with at least three vertices, then $G/e$ is $k$-edge-connected whenever $G$ is $k$-edge-connected. This is an immediate consequence of the fact that every bond of $G/e$ is also a bond of $G$.
\end{proof}


\section{cc-minors of $2$-edge-connected graphs}\label{2-con}

In this section, we determine a set of unavoidable cc-minors in $2$-edge-connected graphs.

\begin{lemma}\label{2con_to_k1}
    Let $G$ be a graph. Then the following hold.
    \begin{itemize}
        \item[(\romannum{1})] If $G$ is $2$-edge-connected and $C$ is a cycle of $G$, then $G/C$ is either a single-vertex graph or is $2$-edge-connected.
        \item[(\romannum{2})] If $G$ is not $2$-edge-connected and $C$ is a cycle of $G$, then $G/C$ is not $2$-edge-connected.
        \item[(\romannum{3})] If $G$ is $2$-edge-connected, then $G/E(G)$, which is isomorphic to $K_1$, is a cc-minor of $G$. 
        \item[(\romannum{4})] If $|V(G)|\geq 2$ and $K_1$ is a cc-minor of $G$, then $G$ is $2$-edge-connected. 
    \end{itemize}
\end{lemma}

\begin{proof}
   If $G$ is $2$-edge-connected, then it has no cut edge. By Lemma~\ref{k-edge-connected}, each cc-minor of $G$ also has no cut edge, thus proving~(\romannum{1}). 
   
   If $G$ is not $2$-edge-connected, then either $G$ is a single-vertex graph and all the cycles of $G$ are loops, or $G$ has a cut edge $e$. In either case, if $C$ is a cycle of $G$, then $G/C$ is not $2$-edge-connected, which proves~(\romannum{2}). 
   
   Suppose that $G$ is $2$-edge-connected and define $G_0 = G$. If $G_i$ contains an edge $e$, then $G_i$ has a cycle $C_i$ that contains $e$. Let $G_{i+1} = G_i / C_i$. This process generates a sequence $G_0, G_1, \dots, G_q$ such that $G_q$ is connected and $E(G_q) = \emptyset$. Hence $G_q \cong K_1$, proving~(\romannum{3}).  
   
   If $G$ has a cut edge, then, by Lemma~\ref{k-edge-connected}, every cc-minor of $G$ will also have a cut edge, and thus it can never be isomorphic to $K_1$. Thus (\romannum{4}) holds.
\end{proof}
Using Lemma~\ref{2con_to_k1}, we define the operation of {\it contracting a $2$-edge-connect\-ed subgraph} $F$ of $G$ as performing a sequence of cycle contractions equivalent to contracting all of the edges in $F$. 

\begin{corollary}\label{cor:contracting subgraph}
    If $F$ is a $2$-edge-connected subgraph of $G$, then $G/E(F)$ is a cc-minor of $G$.
\end{corollary}
Next we prove a characterization of cc-minors.\\

\noindent{\textit{Proof of Proposition}~\ref{prop:structure}.} If (\romannum{1}) and (\romannum{2}) hold, then, by Lemma~\ref{2con_to_k1}(\romannum{3}), $G$ has $H$ as a cc-minor. Note that $H=G/\big(\cup_{v\in V(H)}E(G_v)\big)$. To prove the converse, suppose that $H=G/\{e_1,e_2,\dots,e_k\}$. Let $J$ be the subgraph of $G$ induced by the set $\{e_1,e_2,\dots,e_k\}$ of edges. It suffices to show that each component of $J$ is $2$-edge-connected. Since none of $e_1,e_2,\dots,e_k$ is a loop, each component of $J$ has at least two vertices. Moreover, each component has $K_1$ as a cc-minor and hence, by Lemma~\ref{2con_to_k1}(\romannum{4}), is $2$-edge-connected.\qed\\

The next lemma identifies a set of unavoidable cc-minors in $2$-connected graphs when preserving a specified edge. Let $G_1$ and $G_2$ be graphs such that $V(G_1) \cap V(G_2) = \{u, v\}$ and $E(G_1) \cap E(G_2) = \{e\}$, where $e$ is neither a loop nor a cut edge in $G_1$ or $G_2$. Recall that the graph $G_1 \cup G_2$ is the parallel connection of $G_1$ and $G_2$ with basepoint $e$. More generally, let $G_1, G_2, \dots, G_n$ be a collection of graphs such that, for all distinct $i$ and $j$ in $\{1,2,\dots,n\}$, we have $V(G_i) \cap V(G_j) = \{u, v\}$ and $E(G_i) \cap E(G_j) = \{e\}$, where $e$ is neither a loop nor a cut edge in any $G_k$. The union $G_1 \cup G_2 \cup \dots \cup G_n$ is called the {\it parallel connection} of $G_1, G_2, \dots, G_n$ with basepoint $e$. If $v$ is a vertex of degree two in a graph $G$ and $v$ does not meet a loop of $G$, then, by {\it suppressing v}, we mean deleting $v$ and adding an edge between its two neighbors. 

\begin{lemma}\label{2con_to_parallel}
    Let $G$ be a $2$-edge-connected graph, and let $e$ be a non-loop edge of $G$. Then $G$ has a cc-minor $H$ that is the parallel connection, with basepoint $e$, of a collection of cycles containing $e$.
\end{lemma}

\begin{proof}
    Let $G_0 = G$ and let $e = uv$. If $G_i$ has a cycle $C_i$ that does not contain both $u$ and $v$, define $G_{i+1} = G_i / C_i$. Repeating this process generates a sequence $G_0, G_1, \dots, G_k$ of graphs such that $u$ and $v$ remain distinct vertices in $G_k$, and every cycle in $G_k$ contains both $u$ and $v$. We will show that $G_k$ is a parallel connection, with basepoint $e$, of a collection of cycles containing $e$.
    
    Let $w$ be a vertex of $G_k$ that is not in $\{u,v\}$. We first prove that $d(w) = 2$. By Lemma~\ref{k-edge-connected}, $G_k$ is $2$-edge-connected, so $d(w) \geq 2$. Now, $w$ does not meet a loop of $G_k$. Let $g$ be an edge meeting $w$. Since $g$ is not a cut edge of $G_k$, there is a cycle $C$ in $G_k$ containing $g$, and $C$ must contain both $u$ and $v$. Suppose $d(w) \geq 3$. Then there is an edge $f$ that meets $w$ but is not in $C$. Let $x$ be the other endpoint of $f$. By the choice of $G_k$, we see that $x\notin V(C)$. As $G_k$ is $2$-edge-connected, $G_k - f$ is connected. Choose $P$ as a shortest path in $G_k - f$ from $x$ to a vertex in $C$. Let $P^+$ be the path in $G_k$ that consists of $P$ and the edge $f$. Note that $E(P^+) \cap E(C) = \emptyset$ and $V(P^+) \cap V(C) = \{w, y\}$ for some $y \in V(C)$. Since $P^+$ cannot contain both $u$ and $v$, there is a cycle in $G_k$ that does not contain both $u$ and $v$, a contradiction. Hence $d(w)=2$.

    Let $A$ be the set of degree-$2$ vertices in $G_k$. If we suppress $A-\{u, v\}$, then for some $n\geq 2$, the resulting graph will be a bond graph $B_n$ with vertex set $\{u,v\}$. Hence, $G_k$ is the parallel connection, with basepoint $e$, of a collection of $n-1$ cycles containing $e$.
\end{proof}

Combining Lemma~\ref{k-edge-connected} and Lemma~\ref{2con_to_parallel}, we obtain the following result.
\begin{corollary}\label{cor:3-edge-con to parallel}
    Let $G$ be a $3$-edge-connected graph, and let $e$ be a non-loop edge of $G$. Then $G$ has a cc-minor $H$ that is a bond graph $B_n$ containing $e$, for some $n \geq 3$.
\end{corollary}


\section{Classes closed under cc-minors}\label{sec:classes closed under cc-minor}
Let $\mathcal{F}_1$ be the class of loopless connected graphs. For each positive integer $k > 1$, let $\mathcal{F}_k$ be the class consisting of all loopless $k$-edge-connected graphs along with the single-vertex graph $K_1$. The next proposition follows immediately from Lemma~\ref{k-edge-connected} and Lemma~\ref{2con_to_k1}(\romannum{1}).

\begin{proposition}
    For every positive integer $k$, the class $\mathcal{F}_k$ is closed under cycle contraction.
\end{proposition}

Clearly, $\mathcal{F}_{k+1}\subseteq\mathcal{F}_{k}$ for each positive integer $k$. Now we provide a forbidden cc-minor characterization of $\mathcal{F}_k$ for each $k$ in $\{1,2,3\}$.

\begin{theorem}The following statements hold for a loopless graph $G$.

\begin{enumerate}
    \item[(\romannum{1})] $G$ is in $\mathcal{F}_1$ if and only if $G$ does not have a forest with at least two components as a cc-minor.
    \item[(\romannum{2})] $G$ is in $\mathcal{F}_2$ if and only if $G$ is in $\mathcal{F}_1$ and $G$ does not have a tree with at least one edge as a cc-minor.
    \item[(\romannum{3})] $G$ is in $\mathcal{F}_3$ if and only if $G$ is in $\mathcal{F}_2$ and $G$ does not have a cycle of size at least two as a cc-minor.
\end{enumerate}
\end{theorem} 
\begin{proof}
    First we observe the following.
    \begin{enumerate}
        \item[(a)] A forest with at least two components is not in $\mathcal{F}_1$.
        \item[(b)] A tree with at least one edge is not in $\mathcal{F}_2$.
        \item[(c)] A cycle of size at least two is not in $\mathcal{F}_3$.
    \end{enumerate}
    Since each $\mathcal{F}_k$ is closed under cycle contractions, to prove the theorem, it remains to show that the graphs in (a)-(c) are the only obstructions to membership of $\mathcal{F}_k$ for $k$ in $\{1,2,3\}$.
    
    First, suppose $G$ is not connected. Let $F$ be a cc-minor of $G$ obtained by repeatedly contracting cycles until no cycles remain. Evidently, $F$ is a forest that has the same number of components as $G$. Thus, $F$ is a forest with at least two components, which confirms (\romannum{1}). 
    
    Now, suppose $G\in\mathcal{F}_1-\mathcal{F}_2$. Clearly, $G$ has a cut edge $e$. Let $H$ be a cc-minor of $G$ obtained by repeatedly contracting cycles until no cycles remain. Evidently, $H$ is a tree that contains $e$, which confirms (\romannum{2}).

    Finally, suppose that $G\in\mathcal{F}_2-\mathcal{F}_3$. Since $G$ is not $3$-edge-connected, $G$ has a bond $\{e,f\}$. Let $J$ be a cc-minor of $G$ obtained by repeatedly contracting cycles that contain neither $e$ nor $f$ until no such cycles remain. By Lemma~\ref{k-edge-connected}, $J$ is $2$-edge-connected. Since $\{e,f\}$ is a bond of $J$, every cycle of $J$ contains either none or all of $\{e,f\}$. Therefore, by the construction of $J$, every cycle in $J$ contains both $e$ and $f$. Suppose $C_1$ and $C_2$ are two distinct cycles of $J$. Then $C_1\Delta C_2$, which equals $(C_1\cup C_2)-(C_1\cap C_2)$, is a non-empty disjoint union of cycles of $J$. However, $\{e,f\}\not\subseteq C_1\Delta C_2$, a contradiction. Therefore, $J$ is a cycle, which confirms (\romannum{3}).

\end{proof}

The next theorem characterizes $\mathcal{F}_k$ for all $k>3$.

\begin{theorem}\label{thm: forb_F_k}
    Let $k$ be an integer exceeding three. A loopless graph $G$ is in $\mathcal{F}_{k}$ if and only if $G$ is in $\mathcal{F}_{k-1}$ and  $G$ does not have a cc-minor isomorphic to $B_{k-1}$.
\end{theorem}

\begin{proof}
    Clearly $\mathcal{F}_{k-1} \subseteq \mathcal{F}_{k}$ and $B_{k-1} \notin \mathcal{F}_{k}$. To prove the converse, suppose that $G$ belongs to $\mathcal{F}_{k-1} - \mathcal{F}_{k}$. Let $\{x_1, x_2, \dots, x_{k-1}\}$ be a bond of $G$, and let $H$ be a cc-minor of $G$ obtained by repeatedly contracting cycles that do not contain any edge in $\{x_1, x_2, \dots, x_{k-1}\}$ until no such cycles remain. Evidently, $H \setminus \{x_1, x_2, \dots, x_{k-1}\}$ consists of two components, $T_1$ and $T_2$, each of which is a tree. Since $H$ is $(k-1)$-edge-connected, we have $d_H(v) \geq k-1$ for all $v \in V(H)$. Let $l$ be a leaf of $T_i$ for some $i \in \{1,2\}$. In $H$, the leaf $l$ is incident with at least $k-2$ edges from $\{x_1, x_2, \dots, x_{k-1}\}$. However, for $k > 3$, we have $2(k-2) > k-1$. Therefore, $T_i$ has at most one leaf for each $i \in \{1,2\}$. Thus, both $T_1$ and $T_2$ must be single-vertex graphs, and we conclude that $H \cong B_{k-1}$, which confirms the theorem.

\end{proof}


\section{cc-minors of $3$-connected graphs}\label{3-con}
In this section, we determine a set of unavoidable cc-minors of $3$-connected graphs. By a {\it simplification} of a graph $G$, we mean a simple graph that is obtained from $G$ by deleting all the loops and deleting all but one edge from each maximal set of parallel edges.

\begin{theorem}\label{3-con cc_minor}
    Let $G$ be a simple $3$-connected graph, and let $e$ and $f$ be two distinct edges in $G$. Then $G$ has a cc-minor $H$ containing $e$ and $f$ such that one of the following holds.
    \begin{itemize}
        \item[(\romannum{1})] For some $n\geq 3$, the graph $H$ is isomorphic to a bond graph $B_n$ containing $e$ and $f$; or
        \item[(\romannum{2})] $H$ is isomorphic to a fan-type graph of which $e$ and $f$ are distinct outer spokes that are not parallel; or
        \item[(\romannum{3})] a simplification of $H$ has $e$ and $f$ as non-adjacent edges and is isomorphic to $K_4$. Moreover, if an edge $g$ of $H$ is not parallel to $e$ or $f$, then $g$ is not parallel to any edge in $E(H)$.
    \end{itemize}
\end{theorem}

\begin{figure}[htb]
    \hbox to \hsize{
	\hfil
	\resizebox{8cm}{!}{\tikzset{every picture/.style={line width=0.75pt}} 

\begin{tikzpicture}[x=0.75pt,y=0.75pt,yscale=-1,xscale=1]

\draw [color={rgb, 255:red, 0; green, 0; blue, 0 }  ,draw opacity=1 ] [dash pattern={on 4.5pt off 4.5pt}]  (307.69,67.49) -- (307.69,162.32) ;
\draw [color={rgb, 255:red, 0; green, 0; blue, 0 }  ,draw opacity=1 ] [dash pattern={on 4.5pt off 4.5pt}]  (401.2,67.49) -- (401.2,162.32) ;
\draw  [color={rgb, 255:red, 0; green, 0; blue, 0 }  ,draw opacity=1 ][fill={rgb, 255:red, 0; green, 0; blue, 0 }  ,fill opacity=1 ] (304.5,67.49) .. controls (304.5,65.72) and (305.92,64.3) .. (307.69,64.3) .. controls (309.45,64.3) and (310.88,65.72) .. (310.88,67.49) .. controls (310.88,69.25) and (309.45,70.68) .. (307.69,70.68) .. controls (305.92,70.68) and (304.5,69.25) .. (304.5,67.49) -- cycle ;
\draw  [color={rgb, 255:red, 0; green, 0; blue, 0 }  ,draw opacity=1 ][fill={rgb, 255:red, 0; green, 0; blue, 0 }  ,fill opacity=1 ] (304.5,162.32) .. controls (304.5,160.56) and (305.92,159.13) .. (307.69,159.13) .. controls (309.45,159.13) and (310.88,160.56) .. (310.88,162.32) .. controls (310.88,164.08) and (309.45,165.51) .. (307.69,165.51) .. controls (305.92,165.51) and (304.5,164.08) .. (304.5,162.32) -- cycle ;
\draw  [color={rgb, 255:red, 0; green, 0; blue, 0 }  ,draw opacity=1 ][fill={rgb, 255:red, 0; green, 0; blue, 0 }  ,fill opacity=1 ] (398.01,67.49) .. controls (398.01,65.72) and (399.44,64.3) .. (401.2,64.3) .. controls (402.96,64.3) and (404.39,65.72) .. (404.39,67.49) .. controls (404.39,69.25) and (402.96,70.68) .. (401.2,70.68) .. controls (399.44,70.68) and (398.01,69.25) .. (398.01,67.49) -- cycle ;
\draw  [color={rgb, 255:red, 0; green, 0; blue, 0 }  ,draw opacity=1 ][fill={rgb, 255:red, 0; green, 0; blue, 0 }  ,fill opacity=1 ] (398.01,162.32) .. controls (398.01,160.56) and (399.44,159.13) .. (401.2,159.13) .. controls (402.96,159.13) and (404.39,160.56) .. (404.39,162.32) .. controls (404.39,164.08) and (402.96,165.51) .. (401.2,165.51) .. controls (399.44,165.51) and (398.01,164.08) .. (398.01,162.32) -- cycle ;
\draw [color={rgb, 255:red, 0; green, 0; blue, 0 }  ,draw opacity=1 ]   (307.69,67.49) -- (401.2,67.49) ;
\draw [color={rgb, 255:red, 0; green, 0; blue, 0 }  ,draw opacity=1 ]   (307.69,162.32) -- (401.2,162.32) ;
\draw [color={rgb, 255:red, 0; green, 0; blue, 0 }  ,draw opacity=1 ]   (307.69,67.49) -- (401.2,162.32) ;
\draw [color={rgb, 255:red, 0; green, 0; blue, 0 }  ,draw opacity=1 ]   (307.69,162.32) -- (401.2,67.49) ;
\draw [color={rgb, 255:red, 0; green, 0; blue, 0 }  ,draw opacity=1 ][line width=0.75]  [dash pattern={on 4.5pt off 4.5pt}]  (154.22,70.13) -- (85.38,156.54) ;
\draw  [color={rgb, 255:red, 0; green, 0; blue, 0 }  ,draw opacity=1 ][fill={rgb, 255:red, 0; green, 0; blue, 0 }  ,fill opacity=1 ] (151.03,70.13) .. controls (151.03,68.36) and (152.46,66.93) .. (154.22,66.93) .. controls (155.98,66.93) and (157.41,68.36) .. (157.41,70.13) .. controls (157.41,71.89) and (155.98,73.32) .. (154.22,73.32) .. controls (152.46,73.32) and (151.03,71.89) .. (151.03,70.13) -- cycle ;
\draw  [color={rgb, 255:red, 0; green, 0; blue, 0 }  ,draw opacity=1 ][fill={rgb, 255:red, 0; green, 0; blue, 0 }  ,fill opacity=1 ] (82.19,156.54) .. controls (82.19,154.78) and (83.62,153.35) .. (85.38,153.35) .. controls (87.14,153.35) and (88.57,154.78) .. (88.57,156.54) .. controls (88.57,158.31) and (87.14,159.74) .. (85.38,159.74) .. controls (83.62,159.74) and (82.19,158.31) .. (82.19,156.54) -- cycle ;
\draw  [color={rgb, 255:red, 0; green, 0; blue, 0 }  ,draw opacity=1 ][fill={rgb, 255:red, 0; green, 0; blue, 0 }  ,fill opacity=1 ] (109.73,156.54) .. controls (109.73,154.78) and (111.16,153.35) .. (112.92,153.35) .. controls (114.68,153.35) and (116.11,154.78) .. (116.11,156.54) .. controls (116.11,158.31) and (114.68,159.74) .. (112.92,159.74) .. controls (111.16,159.74) and (109.73,158.31) .. (109.73,156.54) -- cycle ;
\draw  [color={rgb, 255:red, 0; green, 0; blue, 0 }  ,draw opacity=1 ][fill={rgb, 255:red, 0; green, 0; blue, 0 }  ,fill opacity=1 ] (137.27,156.54) .. controls (137.27,154.78) and (138.7,153.35) .. (140.46,153.35) .. controls (142.22,153.35) and (143.65,154.78) .. (143.65,156.54) .. controls (143.65,158.31) and (142.22,159.74) .. (140.46,159.74) .. controls (138.7,159.74) and (137.27,158.31) .. (137.27,156.54) -- cycle ;
\draw  [color={rgb, 255:red, 0; green, 0; blue, 0 }  ,draw opacity=1 ][fill={rgb, 255:red, 0; green, 0; blue, 0 }  ,fill opacity=1 ] (164.81,156.54) .. controls (164.81,154.78) and (166.24,153.35) .. (168,153.35) .. controls (169.76,153.35) and (171.19,154.78) .. (171.19,156.54) .. controls (171.19,158.31) and (169.76,159.74) .. (168,159.74) .. controls (166.24,159.74) and (164.81,158.31) .. (164.81,156.54) -- cycle ;
\draw  [color={rgb, 255:red, 0; green, 0; blue, 0 }  ,draw opacity=1 ][fill={rgb, 255:red, 0; green, 0; blue, 0 }  ,fill opacity=1 ] (192.35,156.54) .. controls (192.35,154.78) and (193.78,153.35) .. (195.54,153.35) .. controls (197.3,153.35) and (198.73,154.78) .. (198.73,156.54) .. controls (198.73,158.31) and (197.3,159.74) .. (195.54,159.74) .. controls (193.78,159.74) and (192.35,158.31) .. (192.35,156.54) -- cycle ;
\draw  [color={rgb, 255:red, 0; green, 0; blue, 0 }  ,draw opacity=1 ][fill={rgb, 255:red, 0; green, 0; blue, 0 }  ,fill opacity=1 ] (219.87,156.54) .. controls (219.87,154.78) and (221.3,153.35) .. (223.06,153.35) .. controls (224.83,153.35) and (226.26,154.78) .. (226.26,156.54) .. controls (226.26,158.31) and (224.83,159.74) .. (223.06,159.74) .. controls (221.3,159.74) and (219.87,158.31) .. (219.87,156.54) -- cycle ;

\draw [color={rgb, 255:red, 0; green, 0; blue, 0 }  ,draw opacity=1 ]   (154.22,70.13) -- (112.92,156.54) ;
\draw [color={rgb, 255:red, 0; green, 0; blue, 0 }  ,draw opacity=1 ]   (154.22,70.13) -- (140.46,156.54) ;
\draw [color={rgb, 255:red, 0; green, 0; blue, 0 }  ,draw opacity=1 ]   (154.22,70.13) -- (168,156.54) ;
\draw [color={rgb, 255:red, 0; green, 0; blue, 0 }  ,draw opacity=1 ]   (154.22,70.13) -- (195.54,156.54) ;
\draw [color={rgb, 255:red, 0; green, 0; blue, 0 }  ,draw opacity=1 ][line width=0.75]  [dash pattern={on 4.5pt off 4.5pt}]  (154.22,70.13) -- (223.06,156.54) ;
\draw [color={rgb, 255:red, 0; green, 0; blue, 0 }  ,draw opacity=1 ]   (85.38,156.54) -- (112.92,156.54) ;
\draw [color={rgb, 255:red, 0; green, 0; blue, 0 }  ,draw opacity=1 ]   (112.92,156.54) -- (140.46,156.54) ;
\draw [color={rgb, 255:red, 0; green, 0; blue, 0 }  ,draw opacity=1 ]   (168,156.54) -- (195.54,156.54) ;
\draw [color={rgb, 255:red, 0; green, 0; blue, 0 }  ,draw opacity=1 ]   (196.54,156.54) -- (224.06,156.54) ;

\draw (283,111.12) node [anchor=north west][inner sep=0.75pt]  [color={rgb, 255:red, 0; green, 0; blue, 0 }  ,opacity=1 ]  {$e$};
\draw (412,111.12) node [anchor=north west][inner sep=0.75pt]  [color={rgb, 255:red, 0; green, 0; blue, 0 }  ,opacity=1 ]  {$f$};
\draw (138.57,152.91) node [anchor=north west][inner sep=0.75pt]    {$\cdots $};
\draw (100,103.72) node [anchor=north west][inner sep=0.75pt]  [color={rgb, 255:red, 0; green, 0; blue, 0 }  ,opacity=1 ]  {$e$};
\draw (201,102.72) node [anchor=north west][inner sep=0.75pt]  [color={rgb, 255:red, 0; green, 0; blue, 0 }  ,opacity=1 ]  {$f$};

\end{tikzpicture}}%
	\hfil
    }
    \caption{Possible simplifications of $H$ that fall under cases (\romannum{2}) and (\romannum{3}).}\label{fig:K4_and_Fan}
    \end{figure}
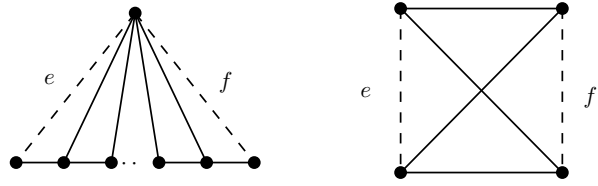

\begin{proof}

Let $e$ have ends $x_1$ and $x_2$, and let $f$ have ends $y_1$ and $y_2$. First, we prove the following.

\begin{sublemma}\label{e,f_adjacent}
    If $e$ and $f$ are adjacent, then $G$ has a cc-minor $H$ such that $H$ is isomorphic to $B_n$, for some $n\geq 3$, and $H$ has $e$ and $f$ as distinct edges.
\end{sublemma}

We may assume that $x_1 = y_1$ and $x_2 \neq y_2$. Since $G$ is $3$-connected, $G - x_1$ is $2$-connected and hence 2-edge-connected. By Lemma~\ref{2con_to_k1} (\romannum{3}), contracting $E(G - x_1)$ forms a cc-minor that is isomorphic to a bond graph that contains $e$ and $f$. Moreover, by Lemma~\ref{k-edge-connected}, $H$ has at least three edges. Hence, \ref{e,f_adjacent} holds.\\

Now we assume $e$ and $f$ are not adjacent. For two distinct vertices $u$ and $v$, a \textit{$\Theta$-graph on $(u, v)$} consists of three internally disjoint $uv$-paths. For each $(x, y) \in \{x_1, x_2\} \times \{y_1, y_2\}$, a $\Theta$-graph on $(x, y)$ that contains both $e$ and $f$ is classified as \textit{type-A} if one of its $xy$-paths contains both $e$ and $f$. It is \textit{type-B} if $e$ and $f$ belong to two different $xy$-paths of the $\Theta$-graph. Examples of these types of graphs are shown in Figure~\ref{fig:special_theta}, where $e$ and $f$ represent single edges, while the other lines in the diagram correspond to paths.

\begin{figure}[htb]
\hbox to \hsize{
\hfil
\resizebox{10cm}{!}{\tikzset{every picture/.style={line width=0.75pt}} 

\begin{tikzpicture}[x=0.75pt,y=0.75pt,yscale=-1,xscale=1]

\draw [color={rgb, 255:red, 0; green, 0; blue, 0 }  ,draw opacity=1 ] [dash pattern={on 4.5pt off 4.5pt}]  (200.69,114.56) -- (235.69,153.56) ;
\draw [color={rgb, 255:red, 0; green, 0; blue, 0 }  ,draw opacity=1 ] [dash pattern={on 4.5pt off 4.5pt}]  (117.69,114.56) -- (82.69,153.56) ;
\draw  [color={rgb, 255:red, 0; green, 0; blue, 0 }  ,draw opacity=1 ][fill={rgb, 255:red, 0; green, 0; blue, 0 }  ,fill opacity=1 ] (79.5,153.56) .. controls (79.5,151.8) and (80.92,150.37) .. (82.69,150.37) .. controls (84.45,150.37) and (85.88,151.8) .. (85.88,153.56) .. controls (85.88,155.33) and (84.45,156.75) .. (82.69,156.75) .. controls (80.92,156.75) and (79.5,155.33) .. (79.5,153.56) -- cycle ;
\draw  [color={rgb, 255:red, 0; green, 0; blue, 0 }  ,draw opacity=1 ][fill={rgb, 255:red, 0; green, 0; blue, 0 }  ,fill opacity=1 ] (114.5,114.56) .. controls (114.5,112.8) and (115.92,111.37) .. (117.69,111.37) .. controls (119.45,111.37) and (120.88,112.8) .. (120.88,114.56) .. controls (120.88,116.33) and (119.45,117.75) .. (117.69,117.75) .. controls (115.92,117.75) and (114.5,116.33) .. (114.5,114.56) -- cycle ;

\draw  [color={rgb, 255:red, 0; green, 0; blue, 0 }  ,draw opacity=1 ][fill={rgb, 255:red, 0; green, 0; blue, 0 }  ,fill opacity=1 ] (197.5,114.56) .. controls (197.5,116.33) and (198.92,117.75) .. (200.69,117.75) .. controls (202.45,117.75) and (203.88,116.33) .. (203.88,114.56) .. controls (203.88,112.8) and (202.45,111.37) .. (200.69,111.37) .. controls (198.92,111.37) and (197.5,112.8) .. (197.5,114.56) -- cycle ;
\draw  [color={rgb, 255:red, 0; green, 0; blue, 0 }  ,draw opacity=1 ][fill={rgb, 255:red, 0; green, 0; blue, 0 }  ,fill opacity=1 ] (232.5,153.56) .. controls (232.5,155.33) and (233.92,156.75) .. (235.69,156.75) .. controls (237.45,156.75) and (238.88,155.33) .. (238.88,153.56) .. controls (238.88,151.8) and (237.45,150.37) .. (235.69,150.37) .. controls (233.92,150.37) and (232.5,151.8) .. (232.5,153.56) -- cycle ;

\draw [color={rgb, 255:red, 0; green, 0; blue, 0 }  ,draw opacity=1 ]   (117.69,114.56) .. controls (136.2,96.34) and (183.2,95.34) .. (200.69,114.56) ;
\draw [color={rgb, 255:red, 0; green, 0; blue, 0 }  ,draw opacity=1 ]   (82.69,153.56) .. controls (112.2,219.34) and (204.39,220.34) .. (235.69,153.56) ;
\draw [color={rgb, 255:red, 0; green, 0; blue, 0 }  ,draw opacity=1 ]   (82.69,153.56) -- (235.69,153.56) ;
\draw [color={rgb, 255:red, 0; green, 0; blue, 0 }  ,draw opacity=1 ] [dash pattern={on 4.5pt off 4.5pt}]  (529.81,144.82) -- (494.81,183.82) ;
\draw [color={rgb, 255:red, 0; green, 0; blue, 0 }  ,draw opacity=1 ] [dash pattern={on 4.5pt off 4.5pt}]  (423.69,106.82) -- (388.69,145.82) ;
\draw  [fill={rgb, 255:red, 0; green, 0; blue, 0 }  ,fill opacity=1 ] (385.5,145.82) .. controls (385.5,144.06) and (386.92,142.63) .. (388.69,142.63) .. controls (390.45,142.63) and (391.88,144.06) .. (391.88,145.82) .. controls (391.88,147.58) and (390.45,149.01) .. (388.69,149.01) .. controls (386.92,149.01) and (385.5,147.58) .. (385.5,145.82) -- cycle ;
\draw  [fill={rgb, 255:red, 0; green, 0; blue, 0 }  ,fill opacity=1 ] (420.5,106.82) .. controls (420.5,105.06) and (421.92,103.63) .. (423.69,103.63) .. controls (425.45,103.63) and (426.88,105.06) .. (426.88,106.82) .. controls (426.88,108.58) and (425.45,110.01) .. (423.69,110.01) .. controls (421.92,110.01) and (420.5,108.58) .. (420.5,106.82) -- cycle ;

\draw  [fill={rgb, 255:red, 0; green, 0; blue, 0 }  ,fill opacity=1 ] (533,144.82) .. controls (533,146.58) and (531.57,148.01) .. (529.81,148.01) .. controls (528.05,148.01) and (526.62,146.58) .. (526.62,144.82) .. controls (526.62,143.06) and (528.05,141.63) .. (529.81,141.63) .. controls (531.57,141.63) and (533,143.06) .. (533,144.82) -- cycle ;
\draw  [fill={rgb, 255:red, 0; green, 0; blue, 0 }  ,fill opacity=1 ] (498,183.82) .. controls (498,185.58) and (496.57,187.01) .. (494.81,187.01) .. controls (493.05,187.01) and (491.62,185.58) .. (491.62,183.82) .. controls (491.62,182.06) and (493.05,180.63) .. (494.81,180.63) .. controls (496.57,180.63) and (498,182.06) .. (498,183.82) -- cycle ;

\draw    (388.69,145.82) -- (529.81,144.82) ;
\draw    (423.69,106.82) .. controls (463.69,76.82) and (528.2,99.6) .. (529.81,144.82) ;
\draw    (388.69,145.82) .. controls (378.2,203.6) and (456.2,214.6) .. (494.81,183.82) ;

\draw (84,124.86) node [anchor=north west][inner sep=0.75pt]  [color={rgb, 255:red, 0; green, 0; blue, 0 }  ,opacity=1 ]  {$e$};
\draw (224,123.86) node [anchor=north west][inner sep=0.75pt]  [color={rgb, 255:red, 0; green, 0; blue, 0 }  ,opacity=1 ]  {$f$};
\draw (145,216.4) node [anchor=north west][inner sep=0.75pt]    {$\Theta _{A}$};
\draw (61,150.14) node [anchor=north west][inner sep=0.75pt]    {$x$};
\draw (242,150.14) node [anchor=north west][inner sep=0.75pt]    {$y$};
\draw (367,142.4) node [anchor=north west][inner sep=0.75pt]    {$x$};
\draw (540,139.4) node [anchor=north west][inner sep=0.75pt]    {$y$};
\draw (388,119.12) node [anchor=north west][inner sep=0.75pt]  [color={rgb, 255:red, 0; green, 0; blue, 0 }  ,opacity=1 ]  {$e$};
\draw (521,170.12) node [anchor=north west][inner sep=0.75pt]  [color={rgb, 255:red, 0; green, 0; blue, 0 }  ,opacity=1 ]  {$f$};
\draw (435,215.4) node [anchor=north west][inner sep=0.75pt]    {$\Theta _{B}$};

\end{tikzpicture}}%
\hfil
}
\caption{A type-A $\Theta$-graph $\Theta_A$ and a type-B $\Theta$-graph $\Theta_B$}\label{fig:special_theta}
\end{figure}
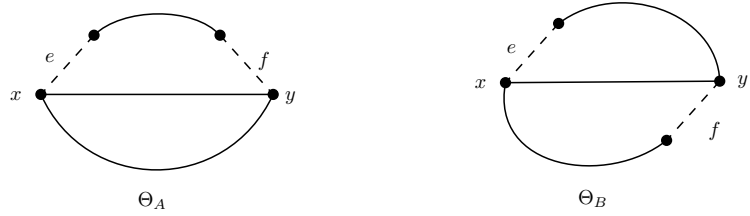

Next we prove the following.

\begin{sublemma}\label{theta-graph}
    For each pair $(x, y) \in \{x_1, x_2\} \times \{y_1, y_2\}$, there is a $\Theta$-graph on $(x, y)$ that contains $e$ and $f$.
\end{sublemma}

Without loss of generality, assume $x = x_1$ and $y = y_1$. By Menger's Theorem, there are three internally disjoint $x_1y_1$-paths that form a $\Theta$-graph on $(x,y)$. Choose $\Theta_0$ to be a $\Theta$-graph on $(x,y)$ that contains the maximal number of members of $\{e,f\}$. We may assume that $e\notin E(\Theta_0)$. Since $G - x_1$ is connected, there is an $x_2y_1$-path in $G - x_1$ whose vertices and edges, in order, are $v_1e_1v_2 \dots e_{k-1}v_k$, where $v_1 = x_2$ and $v_k = y_1$. In $G$, we adjoin the edge $e$ to the beginning of this path to form a path $P$. Let $i$ be the smallest index such that $v_i \in V(\Theta_0)$.

Suppose $v_i \neq y_1$. Then $v_i$ lies on an $x_1y_1$-path $Q$ in $\Theta_0$. Note that the $x_1v_i$-subpath $Q'$ of $Q$ does not use the edge $f$. Replacing $Q'$ by the $x_1v_i$-subpath of $P$, we obtain a $\Theta$-graph that violates the choice of $\Theta_0$. Thus $v_i = y_1$. Choose an $x_1y_1$-path $R$ in $\Theta_0$ such that $f \notin E(R)$, and let $\Theta_1$ be the $\Theta$-graph obtained by replacing $R$ with $P$. Then $\Theta_1$ violates the choice of $\Theta_0$. The contradiction completes the proof of~\ref{theta-graph}.\\

Next we show that if $G$ contains a type-A $\Theta$-graph, then the theorem holds.

\begin{sublemma}\label{type-A}
    If $G$ has a cc-minor $G'$ that contains a type-A $\Theta$-graph $\Theta_A$ as a subgraph, then $G$ has a cc-minor $H$ containing $e$ and $f$ such that $H$ is either isomorphic to a fan-type graph, with $e$ and $f$ as distinct non-parallel outer spokes, or $H$ is isomorphic to $B_n$ for some $n \geq 3$.
\end{sublemma}

By Lemma~\ref{k-edge-connected}, every cc-minor of $G$ is $3$-edge-connected. Thus, we may assume that $\Theta_A$ is $3$-edge-connected. Without loss of generality, we may also assume that $\Theta_A$ is a $\Theta$-graph on $(x_1,y_1)$. Observe that $\Theta_A$ has a cycle $C$ that does not contain $e$ or $f$. Contracting $C$ in $G'$ has the effect of identifying all the vertices in $V(C)$ as a single vertex $c$ and then deleting all the edges in $E(C)$. Since $e$ and $f$ are contained in a cycle of $G/C$, there is a maximal $2$-connected subgraph $L$ of $G'/C$ containing both $e$ and $f$. 

Since $G'/C$ has no cut edges, every edge of $G'/C$ is in a 2-connected subgraph of $G'/C$.  Now, by Lemma~\ref{2con_to_k1} (\romannum{3}), we may contract the edges of all of the maximal $2$-connected subgraphs of $G'/C$ except $L$. The resulting graph is isomorphic to $L$, so we continue referring to it as $L$. Let $L_0 = L$. For $i\geq 0$, if $L_i^-$ contains a cycle $C_i$, let $L_{i+1} = L_i / C_i$. This process produces a sequence $L_0, L_1, \dots, L_s$ of graphs such that $L_s^-$ is a tree $T$. Let $P$ be the $x_2y_2$-path in $T$. Note that $T$ may be a single vertex if $x_2$ and $y_2$  have been identified. 

For $j\geq s$, if $L_j^- \neq P$, then there is a leaf $l_j \in V(L_j^-) - V(P)$. By Lemma~\ref{k-edge-connected}, $d(l_j) \geq 3$, so there is a cycle $O_j$ consisting of two edges with ends $c$ and $l_j$. Let $L_{j+1} = L_j / O_j$. This results in a sequence $L_s, L_{s+1}, \dots, L_{s+t}$ such that $L_{s+t} - c = P$. In $L_{s+t}$, since each vertex of $P$ has degree at least three, there is at least one $cp$-edge for each $p \in V(P)$. Thus, $L_{s+t}$ is a fan-type graph with $e$ and $f$ as two outer spokes. Moreover, $e$ and $f$ are not parallel unless $L_{s+t}$ is a bond graph with at least three edges. Hence~\ref{type-A} holds.\\

In view of~\ref{type-A}, we may now assume that 
\begin{sublemma}\label{assumption}
$G$ has no cc-minor that contains a type-A $\Theta$-graph.
\end{sublemma}

By~\ref{theta-graph}, $G$ has a type-B $\Theta$-graph $\Theta_B$ as a subgraph. Without loss of generality, suppose $\Theta_B$ is on $(x_1, y_1)$. Throughout the following argument, for any type-B $\Theta$-graph on $(x_1, y_1)$, we denote the $x_2y_1$-path as $P_1$, the $x_1y_1$-path as $P_2$, and the $x_1y_2$-path as $P_3$, as shown in Figure~\ref{fig:theta_B}.

\begin{figure}[htb]
\hbox to \hsize{
\hfil
\resizebox{5cm}{!}{\tikzset{every picture/.style={line width=0.75pt}} 

\begin{tikzpicture}[x=0.75pt,y=0.75pt,yscale=-1,xscale=1]

\draw [color={rgb, 255:red, 0; green, 0; blue, 0 }  ,draw opacity=1 ][dash pattern={on 4.5pt off 4.5pt}]   (102.69,135.56) -- (101.69,188.37) ;
\draw  [fill={rgb, 255:red, 0; green, 0; blue, 0 }  ,fill opacity=1 ] (98.5,191.56) .. controls (98.5,189.8) and (99.92,188.37) .. (101.69,188.37) .. controls (103.45,188.37) and (104.88,189.8) .. (104.88,191.56) .. controls (104.88,193.33) and (103.45,194.75) .. (101.69,194.75) .. controls (99.92,194.75) and (98.5,193.33) .. (98.5,191.56) -- cycle ;
\draw  [fill={rgb, 255:red, 0; green, 0; blue, 0 }  ,fill opacity=1 ] (99.5,135.56) .. controls (99.5,133.8) and (100.92,132.37) .. (102.69,132.37) .. controls (104.45,132.37) and (105.88,133.8) .. (105.88,135.56) .. controls (105.88,137.33) and (104.45,138.75) .. (102.69,138.75) .. controls (100.92,138.75) and (99.5,137.33) .. (99.5,135.56) -- cycle ;

\draw [color={rgb, 255:red, 0; green, 0; blue, 0 }  ,draw opacity=1 ][dash pattern={on 4.5pt off 4.5pt}]   (221.69,135.56) -- (220.69,188.37) ;
\draw  [fill={rgb, 255:red, 0; green, 0; blue, 0 }  ,fill opacity=1 ] (217.5,191.56) .. controls (217.5,189.8) and (218.92,188.37) .. (220.69,188.37) .. controls (222.45,188.37) and (223.88,189.8) .. (223.88,191.56) .. controls (223.88,193.33) and (222.45,194.75) .. (220.69,194.75) .. controls (218.92,194.75) and (217.5,193.33) .. (217.5,191.56) -- cycle ;
\draw  [fill={rgb, 255:red, 0; green, 0; blue, 0 }  ,fill opacity=1 ] (218.5,135.56) .. controls (218.5,133.8) and (219.92,132.37) .. (221.69,132.37) .. controls (223.45,132.37) and (224.88,133.8) .. (224.88,135.56) .. controls (224.88,137.33) and (223.45,138.75) .. (221.69,138.75) .. controls (219.92,138.75) and (218.5,137.33) .. (218.5,135.56) -- cycle ;

\draw    (102.69,135.56) -- (221.69,135.56) ;
\draw    (101.69,191.56) -- (220.69,191.56) ;
\draw    (101.69,191.56) -- (221.69,135.56) ;

\draw (231.5,154.86) node [anchor=north west][inner sep=0.75pt]   {$f$};
\draw (81.5,156.86) node [anchor=north west][inner sep=0.75pt]   {$e$};
\draw (154,121.4) node [anchor=north west][inner sep=0.75pt]    {$P_{1}$};
\draw (141,151.4) node [anchor=north west][inner sep=0.75pt]    {$P_{2}$};
\draw (148,199.4) node [anchor=north west][inner sep=0.75pt]    {$P_{3}$};
\draw (86,198.4) node [anchor=north west][inner sep=0.75pt]    {$x_{1}$};
\draw (84,118.4) node [anchor=north west][inner sep=0.75pt]    {$x_{2}$};
\draw (222.69,199.96) node [anchor=north west][inner sep=0.75pt]    {$y_{2}$};
\draw (225.69,121.96) node [anchor=north west][inner sep=0.75pt]    {$y_{1}$};

\end{tikzpicture}}%
\hfil
}
\caption{$\Theta_B$}\label{fig:theta_B}
\end{figure}

A graph is a {\it four-path connector} if it consists of the edges $e$ and $f$, along with four internally disjoint paths, $P_1$, $P_2$, $P_3$, and $P_4$, that connect the vertex pairs $\{x_2,y_1\}$, $\{x_1,y_1\}$, $\{x_1,y_2\}$, and $\{x_2,y_2\}$, respectively. Moreover, each path $P_i$ contains at least one edge for every $i \in \{1,2,3,4\}$. Throughout the remainder of this proof, in any four-path connector, the labels of these four paths will be consistent with Figure~\ref{fig:k4 subdivision}. Next, we show the following.

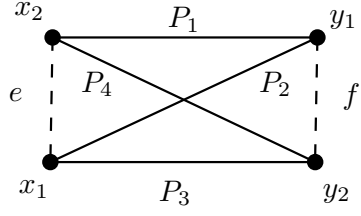
\begin{figure}[htb]
\hbox to \hsize{
\hfil
\resizebox{5cm}{!}{\tikzset{every picture/.style={line width=0.75pt}} 

\begin{tikzpicture}[x=0.75pt,y=0.75pt,yscale=-1,xscale=1]

\draw [color={rgb, 255:red, 0; green, 0; blue, 0 }  ,draw opacity=1 ] [dash pattern={on 4.5pt off 4.5pt}]  (109.69,109.56) -- (108.69,165.56) ;
\draw  [fill={rgb, 255:red, 0; green, 0; blue, 0 }  ,fill opacity=1 ] (105.5,165.56) .. controls (105.5,163.8) and (106.92,162.37) .. (108.69,162.37) .. controls (110.45,162.37) and (111.88,163.8) .. (111.88,165.56) .. controls (111.88,167.33) and (110.45,168.75) .. (108.69,168.75) .. controls (106.92,168.75) and (105.5,167.33) .. (105.5,165.56) -- cycle ;
\draw  [fill={rgb, 255:red, 0; green, 0; blue, 0 }  ,fill opacity=1 ] (106.5,109.56) .. controls (106.5,107.8) and (107.92,106.37) .. (109.69,106.37) .. controls (111.45,106.37) and (112.88,107.8) .. (112.88,109.56) .. controls (112.88,111.33) and (111.45,112.75) .. (109.69,112.75) .. controls (107.92,112.75) and (106.5,111.33) .. (106.5,109.56) -- cycle ;
\draw [color={rgb, 255:red, 0; green, 0; blue, 0 }  ,draw opacity=1 ] [dash pattern={on 4.5pt off 4.5pt}]  (228.69,109.56) -- (227.69,162.37) ;
\draw  [fill={rgb, 255:red, 0; green, 0; blue, 0 }  ,fill opacity=1 ] (224.5,165.56) .. controls (224.5,163.8) and (225.92,162.37) .. (227.69,162.37) .. controls (229.45,162.37) and (230.88,163.8) .. (230.88,165.56) .. controls (230.88,167.33) and (229.45,168.75) .. (227.69,168.75) .. controls (225.92,168.75) and (224.5,167.33) .. (224.5,165.56) -- cycle ;
\draw  [fill={rgb, 255:red, 0; green, 0; blue, 0 }  ,fill opacity=1 ] (225.5,109.56) .. controls (225.5,107.8) and (226.92,106.37) .. (228.69,106.37) .. controls (230.45,106.37) and (231.88,107.8) .. (231.88,109.56) .. controls (231.88,111.33) and (230.45,112.75) .. (228.69,112.75) .. controls (226.92,112.75) and (225.5,111.33) .. (225.5,109.56) -- cycle ;
\draw    (109.69,109.56) -- (228.69,109.56) ;
\draw    (108.69,165.56) -- (227.69,165.56) ;
\draw    (108.69,165.56) -- (228.69,109.56) ;
\draw    (109.69,109.56) -- (227.69,165.56) ;

\draw (93,172.4) node [anchor=north west][inner sep=0.75pt]    {$x_{1}$};
\draw (91,92.4) node [anchor=north west][inner sep=0.75pt]    {$x_{2}$};
\draw (229.69,173.96) node [anchor=north west][inner sep=0.75pt]    {$y_{2}$};
\draw (232.69,95.96) node [anchor=north west][inner sep=0.75pt]    {$y_{1}$};
\draw (88.5,130.86) node [anchor=north west][inner sep=0.75pt]  [color={rgb, 255:red, 0; green, 0; blue, 0 }  ,opacity=1 ]  {$e$};
\draw (160,94.4) node [anchor=north west][inner sep=0.75pt]    {$P_{1}$};
\draw (156,172.4) node [anchor=north west][inner sep=0.75pt]    {$P_{3}$};
\draw (201,123.4) node [anchor=north west][inner sep=0.75pt]    {$P_{2}$};
\draw (121,123.4) node [anchor=north west][inner sep=0.75pt]    {$P_{4}$};
\draw (238.5,128.86) node [anchor=north west][inner sep=0.75pt]  [color={rgb, 255:red, 0; green, 0; blue, 0 }  ,opacity=1 ]  {$f$};

\end{tikzpicture}}%
\hfil
}
\caption{A four-path connector}\label{fig:k4 subdivision}
\end{figure}

\begin{sublemma}\label{sublem:topo_K4}
The graph $G$ has a cc-minor $G_s$ that contains a four-path connector $K$ as a subgraph.
\end{sublemma}

Let $(G_0, \Theta_0) = (G, \Theta_B)$. For $i\geq 0$, define $Q_i$ to be a path in $G_i$ such that

\begin{enumerate}
    \item[(\romannum{1})] $Q_i$ connects $x_2$ to a vertex $v_i$ on $P_1$ such that all vertices of $Q_i$, except for its two ends, are not in $\Theta_i$, or
    \item[(\romannum{2})] $Q_i$ connects $y_2$ to a vertex $v_i$ on $P_3$ such that all vertices of $Q_i$, except for its two ends, are not in $\Theta_i$.
\end{enumerate}

If some $Q_i$ satisfying (\romannum{1}) exists, then let $C_i$ be the cycle formed by $Q_i$ and the $x_2v_i$-subpath of $P_1$. Similarly, if some $Q_i$ satisfying (\romannum{2}) exists, then let $C_i$ be the cycle formed by $Q_i$ and the $y_2v_i$-subpath of $P_3$. Define $(G_{i+1}, \Theta_{i+1}) = (G_i / C_i, \Theta_i / (E(\Theta_i) \cap E(C_i)))$. This process produces a sequence $(G_0, \Theta_0), (G_1, \Theta_1),\\ \dots, (G_s, \Theta_s)$ such that, in $G_s$, no path satisfies the condition that defines $Q_i$. Note that, in this process, $Q_i$ is never an $x_2y_1$-path; otherwise, we would obtain a type-A $\Theta$-graph on $(x_2,y_1)$ having as its paths, $Q_i$, $P_1$, and the path with edge set $\{e,f\}\cup E(P_3)$, which contradicts~\ref{assumption}. Thus, in $G_s$, the path $P_1$ retains at least one edge. Similarly, in $G_s$, the path $P_3$ also retains at least one edge.

\begin{figure}[htb]
\hbox to \hsize{
\hfil
\resizebox{5cm}{!}{\tikzset{every picture/.style={line width=0.75pt}} 

\begin{tikzpicture}[x=0.75pt,y=0.75pt,yscale=-1,xscale=1]

\draw [color={rgb, 255:red, 0; green, 0; blue, 0 }  ,draw opacity=1 ] [dash pattern={on 4.5pt off 4.5pt}]  (89.69,80.56) -- (88.69,136.56) ;
\draw  [fill={rgb, 255:red, 0; green, 0; blue, 0 }  ,fill opacity=1 ] (85.5,136.56) .. controls (85.5,134.8) and (86.92,133.37) .. (88.69,133.37) .. controls (90.45,133.37) and (91.88,134.8) .. (91.88,136.56) .. controls (91.88,138.33) and (90.45,139.75) .. (88.69,139.75) .. controls (86.92,139.75) and (85.5,138.33) .. (85.5,136.56) -- cycle ;
\draw  [fill={rgb, 255:red, 0; green, 0; blue, 0 }  ,fill opacity=1 ] (86.5,80.56) .. controls (86.5,78.8) and (87.92,77.37) .. (89.69,77.37) .. controls (91.45,77.37) and (92.88,78.8) .. (92.88,80.56) .. controls (92.88,82.33) and (91.45,83.75) .. (89.69,83.75) .. controls (87.92,83.75) and (86.5,82.33) .. (86.5,80.56) -- cycle ;
\draw [color={rgb, 255:red, 0; green, 0; blue, 0 }  ,draw opacity=1 ] [dash pattern={on 4.5pt off 4.5pt}]  (208.69,80.56) -- (207.69,136.56) ;
\draw  [color={rgb, 255:red, 0; green, 0; blue, 0 }  ,draw opacity=1 ][fill={rgb, 255:red, 0; green, 0; blue, 0 }  ,fill opacity=1 ] (204.5,136.56) .. controls (204.5,134.8) and (205.92,133.37) .. (207.69,133.37) .. controls (209.45,133.37) and (210.88,134.8) .. (210.88,136.56) .. controls (210.88,138.33) and (209.45,139.75) .. (207.69,139.75) .. controls (205.92,139.75) and (204.5,138.33) .. (204.5,136.56) -- cycle ;
\draw  [color={rgb, 255:red, 0; green, 0; blue, 0 }  ,draw opacity=1 ][fill={rgb, 255:red, 0; green, 0; blue, 0 }  ,fill opacity=1 ] (205.5,80.56) .. controls (205.5,78.8) and (206.92,77.37) .. (208.69,77.37) .. controls (210.45,77.37) and (211.88,78.8) .. (211.88,80.56) .. controls (211.88,82.33) and (210.45,83.75) .. (208.69,83.75) .. controls (206.92,83.75) and (205.5,82.33) .. (205.5,80.56) -- cycle ;
\draw    (89.69,80.56) -- (208.69,80.56) ;
\draw    (88.69,136.56) -- (207.69,136.56) ;
\draw    (88.69,136.56) -- (208.69,80.56) ;
\draw  [fill={rgb, 255:red, 0; green, 0; blue, 0 }  ,fill opacity=1 ] (127.5,80.56) .. controls (127.5,78.8) and (128.92,77.37) .. (130.69,77.37) .. controls (132.45,77.37) and (133.88,78.8) .. (133.88,80.56) .. controls (133.88,82.33) and (132.45,83.75) .. (130.69,83.75) .. controls (128.92,83.75) and (127.5,82.33) .. (127.5,80.56) -- cycle ;
\draw    (89.69,80.56) .. controls (87.2,58.6) and (129.2,48.6) .. (130.69,80.56) ;
\draw    (166.69,136.56) .. controls (167.2,163.6) and (185.2,162.6) .. (191.2,161.6) .. controls (197.2,160.6) and (204.01,157.63) .. (207.69,136.56) ;
\draw  [fill={rgb, 255:red, 0; green, 0; blue, 0 }  ,fill opacity=1 ] (163.5,136.56) .. controls (163.5,134.8) and (164.92,133.37) .. (166.69,133.37) .. controls (168.45,133.37) and (169.88,134.8) .. (169.88,136.56) .. controls (169.88,138.33) and (168.45,139.75) .. (166.69,139.75) .. controls (164.92,139.75) and (163.5,138.33) .. (163.5,136.56) -- cycle ;

\draw (73,143.4) node [anchor=north west][inner sep=0.75pt]    {$x_{1}$};
\draw (71,63.4) node [anchor=north west][inner sep=0.75pt]    {$x_{2}$};
\draw (209.69,144.96) node [anchor=north west][inner sep=0.75pt]    {$y_{2}$};
\draw (212.69,66.96) node [anchor=north west][inner sep=0.75pt]    {$y_{1}$};
\draw (68.5,101.86) node [anchor=north west][inner sep=0.75pt]  [color={rgb, 255:red, 0; green, 0; blue, 0 }  ,opacity=1 ]  {$e$};
\draw (101,45.4) node [anchor=north west][inner sep=0.75pt]    {$Q_{i}$};
\draw (179,162.4) node [anchor=north west][inner sep=0.75pt]    {$Q_{j}$};
\draw (218.5,99.86) node [anchor=north west][inner sep=0.75pt]  [color={rgb, 255:red, 0; green, 0; blue, 0 }  ,opacity=1 ]  {$f$};

\end{tikzpicture}}%
\hfil
}
\caption{Paths similar to $Q_i$ or $Q_j$ will not appear in $G_{s}$.}
\end{figure}
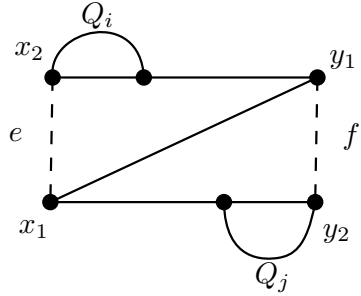

Let $\mathcal{P}_x$ be the collection of paths in $G_s$ that start at $x_2$, end at a vertex in $V(\Theta_s)$, and are \textit{internally disjoint from $\Theta_s$}, meaning they are vertex-disjoint from $\Theta_s$ except at their endpoints. Similarly, let $\mathcal{P}_y$ be the collection of paths that start at $y_2$, end at a vertex in $V(\Theta_s)$, and are internally disjoint from $\Theta_s$.

If there is a path $P \in \mathcal{P}_x$ that ends at a vertex, say $w$, in $V(P_2) - \{x_1\}$, then there is a type-A $\Theta$-graph on $(x_2, y_1)$ (see Figure~\ref{fig:type_A in G_t}), which contradicts~\ref{assumption}. Similarly, we may assume that $\mathcal{P}_y$ does not contain any path that ends at a vertex in $V(P_2) - \{y_1\}$.

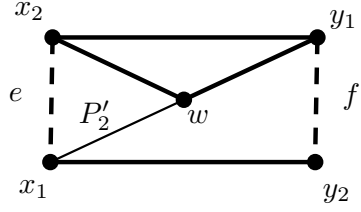
\begin{figure}[htb]
\hbox to \hsize{
\hfil
\resizebox{5cm}{!}{\tikzset{every picture/.style={line width=0.75pt}} 

\begin{tikzpicture}[x=0.75pt,y=0.75pt,yscale=-1,xscale=1]

\draw [color={rgb, 255:red, 0; green, 0; blue, 0 }  ,draw opacity=1 ][line width=1.5]    (109.69,100.56) -- (228.69,100.56) ;
\draw [color={rgb, 255:red, 0; green, 0; blue, 0 }  ,draw opacity=1 ][line width=1.5]    (108.69,156.56) -- (227.69,156.56) ;
\draw [color={rgb, 255:red, 0; green, 0; blue, 0 }  ,draw opacity=1 ][line width=1.5]    (168.69,128.56) -- (228.69,100.56) ;
\draw [color={rgb, 255:red, 0; green, 0; blue, 0 }  ,draw opacity=1 ][line width=1.5]    (109.69,100.56) -- (168.69,128.56) ;
\draw    (108.69,156.56) -- (168.69,128.56) ;
\draw [color={rgb, 255:red, 0; green, 0; blue, 0 }  ,draw opacity=1 ][line width=1.5]  [dash pattern={on 5.63pt off 4.5pt}]  (109.69,100.56) -- (108.69,156.56) ;
\draw  [fill={rgb, 255:red, 0; green, 0; blue, 0 }  ,fill opacity=1 ] (105.5,156.56) .. controls (105.5,154.8) and (106.92,153.37) .. (108.69,153.37) .. controls (110.45,153.37) and (111.88,154.8) .. (111.88,156.56) .. controls (111.88,158.33) and (110.45,159.75) .. (108.69,159.75) .. controls (106.92,159.75) and (105.5,158.33) .. (105.5,156.56) -- cycle ;
\draw  [fill={rgb, 255:red, 0; green, 0; blue, 0 }  ,fill opacity=1 ] (106.5,100.56) .. controls (106.5,98.8) and (107.92,97.37) .. (109.69,97.37) .. controls (111.45,97.37) and (112.88,98.8) .. (112.88,100.56) .. controls (112.88,102.33) and (111.45,103.75) .. (109.69,103.75) .. controls (107.92,103.75) and (106.5,102.33) .. (106.5,100.56) -- cycle ;
\draw [color={rgb, 255:red, 0; green, 0; blue, 0 }  ,draw opacity=1 ][line width=1.5]  [dash pattern={on 5.63pt off 4.5pt}]  (228.69,100.56) -- (227.69,153.37) ;
\draw  [fill={rgb, 255:red, 0; green, 0; blue, 0 }  ,fill opacity=1 ] (224.5,156.56) .. controls (224.5,154.8) and (225.92,153.37) .. (227.69,153.37) .. controls (229.45,153.37) and (230.88,154.8) .. (230.88,156.56) .. controls (230.88,158.33) and (229.45,159.75) .. (227.69,159.75) .. controls (225.92,159.75) and (224.5,158.33) .. (224.5,156.56) -- cycle ;
\draw  [fill={rgb, 255:red, 0; green, 0; blue, 0 }  ,fill opacity=1 ] (225.5,100.56) .. controls (225.5,98.8) and (226.92,97.37) .. (228.69,97.37) .. controls (230.45,97.37) and (231.88,98.8) .. (231.88,100.56) .. controls (231.88,102.33) and (230.45,103.75) .. (228.69,103.75) .. controls (226.92,103.75) and (225.5,102.33) .. (225.5,100.56) -- cycle ;
\draw  [fill={rgb, 255:red, 0; green, 0; blue, 0 }  ,fill opacity=1 ] (165.5,128.56) .. controls (165.5,126.8) and (166.92,125.37) .. (168.69,125.37) .. controls (170.45,125.37) and (171.88,126.8) .. (171.88,128.56) .. controls (171.88,130.33) and (170.45,131.75) .. (168.69,131.75) .. controls (166.92,131.75) and (165.5,130.33) .. (165.5,128.56) -- cycle ;

\draw (93,163.4) node [anchor=north west][inner sep=0.75pt]    {$x_{1}$};
\draw (91,83.4) node [anchor=north west][inner sep=0.75pt]    {$x_{2}$};
\draw (229.69,164.96) node [anchor=north west][inner sep=0.75pt]    {$y_{2}$};
\draw (232.69,86.96) node [anchor=north west][inner sep=0.75pt]    {$y_{1}$};
\draw (88.5,121.86) node [anchor=north west][inner sep=0.75pt]  [color={rgb, 255:red, 0; green, 0; blue, 0 }  ,opacity=1 ]  {$e$};
\draw (168.69,131.96) node [anchor=north west][inner sep=0.75pt]    {$w$};
\draw (120,127.4) node [anchor=north west][inner sep=0.75pt]    {$P_{2} '$};
\draw (238.5,119.86) node [anchor=north west][inner sep=0.75pt]  [color={rgb, 255:red, 0; green, 0; blue, 0 }  ,opacity=1 ]  {$f$};

\end{tikzpicture}}%
\hfil
}
\caption{Deleting the edges of $P_2'$ results in a type-A $\Theta$-graph.}\label{fig:type_A in G_t}
\end{figure}

Note that $G_s$ is obtained from $G$ by repeatedly contracting cycles containing exactly one of $\{x_2,y_2\}$. Because $G$ is $3$-connected, by Lemma~\ref{xy-path}, there are at least three internally disjoint $x_2y_2$-paths in $G_s$. However, if all paths in $\mathcal{P}_x$ end at $x_1$, then every $x_2y_2$-path in $G_s$ contains either $x_1$ or the neighbor of $x_2$ on $P_1$, contradicting the existence of three internally disjoint $x_2y_2$-paths. Therefore, by the choice of $G_s$, we may assume $\mathcal{P}_x$ contains a path $P_x$ that ends at a vertex in $V(P_3) - \{x_1\}$. By symmetry, we may also assume $\mathcal{P}_y$ contains a path $P_y$ that ends at a vertex in $V(P_1) -\{y_1\}$.  If $P_x$ and $P_y$ are internally disjoint, then there is a type-A $\Theta$-graph on $(x_2, y_2)$ (see Figure~\ref{fig:type_A in G_t_2}), which contradicts~\ref{assumption}.

\begin{figure}[htb]
\hbox to \hsize{
\hfil
\resizebox{5cm}{!}{\tikzset{every picture/.style={line width=0.75pt}} 

\begin{tikzpicture}[x=0.75pt,y=0.75pt,yscale=-1,xscale=1]

\draw [color={rgb, 255:red, 0; green, 0; blue, 0 }  ,draw opacity=1 ][line width=1.5]    (109.69,100.56) -- (169.19,100.56) ;
\draw [color={rgb, 255:red, 0; green, 0; blue, 0 }  ,draw opacity=1 ][line width=1.5]    (108.69,156.56) -- (228.69,100.56) ;
\draw [color={rgb, 255:red, 0; green, 0; blue, 0 }  ,draw opacity=1 ][line width=1.5]    (109.69,100.56) -- (168.19,156.56) ;
\draw [color={rgb, 255:red, 0; green, 0; blue, 0 }  ,draw opacity=1 ][line width=1.5]    (169.19,100.56) -- (227.69,156.56) ;
\draw [color={rgb, 255:red, 0; green, 0; blue, 0 }  ,draw opacity=1 ][line width=1.5]    (171.38,156.56) -- (230.88,156.56) ;
\draw [color={rgb, 255:red, 0; green, 0; blue, 0 }  ,draw opacity=1 ][line width=1.5]  [dash pattern={on 5.63pt off 4.5pt}]  (109.69,100.56) -- (108.69,156.56) ;
\draw  [fill={rgb, 255:red, 0; green, 0; blue, 0 }  ,fill opacity=1 ] (105.5,156.56) .. controls (105.5,154.8) and (106.92,153.37) .. (108.69,153.37) .. controls (110.45,153.37) and (111.88,154.8) .. (111.88,156.56) .. controls (111.88,158.33) and (110.45,159.75) .. (108.69,159.75) .. controls (106.92,159.75) and (105.5,158.33) .. (105.5,156.56) -- cycle ;
\draw  [fill={rgb, 255:red, 0; green, 0; blue, 0 }  ,fill opacity=1 ] (106.5,100.56) .. controls (106.5,98.8) and (107.92,97.37) .. (109.69,97.37) .. controls (111.45,97.37) and (112.88,98.8) .. (112.88,100.56) .. controls (112.88,102.33) and (111.45,103.75) .. (109.69,103.75) .. controls (107.92,103.75) and (106.5,102.33) .. (106.5,100.56) -- cycle ;
\draw [color={rgb, 255:red, 0; green, 0; blue, 0 }  ,draw opacity=1 ][line width=1.5]  [dash pattern={on 5.63pt off 4.5pt}]  (228.69,100.56) -- (227.69,153.37) ;
\draw  [color={rgb, 255:red, 0; green, 0; blue, 0 }  ,draw opacity=1 ][fill={rgb, 255:red, 0; green, 0; blue, 0 }  ,fill opacity=1 ] (224.5,156.56) .. controls (224.5,154.8) and (225.92,153.37) .. (227.69,153.37) .. controls (229.45,153.37) and (230.88,154.8) .. (230.88,156.56) .. controls (230.88,158.33) and (229.45,159.75) .. (227.69,159.75) .. controls (225.92,159.75) and (224.5,158.33) .. (224.5,156.56) -- cycle ;
\draw  [color={rgb, 255:red, 0; green, 0; blue, 0 }  ,draw opacity=1 ][fill={rgb, 255:red, 0; green, 0; blue, 0 }  ,fill opacity=1 ] (225.5,100.56) .. controls (225.5,98.8) and (226.92,97.37) .. (228.69,97.37) .. controls (230.45,97.37) and (231.88,98.8) .. (231.88,100.56) .. controls (231.88,102.33) and (230.45,103.75) .. (228.69,103.75) .. controls (226.92,103.75) and (225.5,102.33) .. (225.5,100.56) -- cycle ;
\draw    (108.69,156.56) -- (168.19,156.56) ;
\draw  [fill={rgb, 255:red, 0; green, 0; blue, 0 }  ,fill opacity=1 ] (165,156.56) .. controls (165,154.8) and (166.42,153.37) .. (168.19,153.37) .. controls (169.95,153.37) and (171.38,154.8) .. (171.38,156.56) .. controls (171.38,158.33) and (169.95,159.75) .. (168.19,159.75) .. controls (166.42,159.75) and (165,158.33) .. (165,156.56) -- cycle ;
\draw  [fill={rgb, 255:red, 0; green, 0; blue, 0 }  ,fill opacity=1 ] (166,100.56) .. controls (166,98.8) and (167.42,97.37) .. (169.19,97.37) .. controls (170.95,97.37) and (172.38,98.8) .. (172.38,100.56) .. controls (172.38,102.33) and (170.95,103.75) .. (169.19,103.75) .. controls (167.42,103.75) and (166,102.33) .. (166,100.56) -- cycle ;
\draw    (169.19,100.56) -- (228.69,100.56) ;

\draw (93,163.4) node [anchor=north west][inner sep=0.75pt]    {$x_{1}$};
\draw (91,83.4) node [anchor=north west][inner sep=0.75pt]    {$x_{2}$};
\draw (229.69,164.96) node [anchor=north west][inner sep=0.75pt]    {$y_{2}$};
\draw (232.69,86.96) node [anchor=north west][inner sep=0.75pt]    {$y_{1}$};
\draw (88.5,121.86) node [anchor=north west][inner sep=0.75pt]  [color={rgb, 255:red, 0; green, 0; blue, 0 }  ,opacity=1 ]  {$e$};
\draw (116,119.4) node [anchor=north west][inner sep=0.75pt]    {$P_{x}$};
\draw (207.44,127.96) node [anchor=north west][inner sep=0.75pt]    {$P_{y}$};
\draw (238.5,119.86) node [anchor=north west][inner sep=0.75pt]  [color={rgb, 255:red, 0; green, 0; blue, 0 }  ,opacity=1 ]  {$f$};

\end{tikzpicture}}%
\hfil
}
\caption{The thickened paths form a type-A $\Theta$-graph.}\label{fig:type_A in G_t_2}
\end{figure}
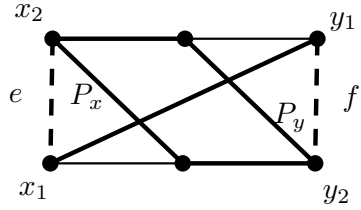
We may now assume that $P_x$ and $P_y$ are not internally disjoint. Therefore, there is an $x_2y_2$-path $P_4$ in $G_{s}$ that is internally disjoint from $\Theta_{s}$. Let $K$ be the subgraph of $G_{s}$ consisting of $\Theta_{s}$ and $P_4$. Hence~\ref{sublem:topo_K4} holds.\\

A four-path connector $F$ in a graph $J$ is {\it spanning} if $V(F)=V(J)$.

\begin{sublemma}\label{sublem:spanning FPC}
    The graph $G_s$ has a cc-minor $G_s'$ that contains a spanning four-path connector $K'$.
\end{sublemma}

It is easy to check that, in the process of obtaining $G_s$ from $G$, every cycle that has been contracted intersects $K$. In order to prove~\ref{sublem:spanning FPC}, it suffices to prove the following.

\begin{sublemma}\label{sublem:absorbtion}
    Let $D$ be a cc-minor of $G$ that contains a non-spanning four-path connector $F$. Moreover, if $D=G/C_1/C_2/\dots/C_n$ where $C_i$ is a cycle of $G/C_1/C_2/\dots/C_{i-1}$ for all $i\in\{1,2,\dots,n\}$, assume that $V(C_i)\cap V(F)\neq\emptyset$ for all such $i$. Then $D$ has a cycle $C_{n+1}$ such that if $D'=D/C_{n+1}$, then $D'$ has a four-path connector $F'$ for which $|V(D)-V(F)|>|V(D')-V(F')|$. In particular, $V(C_{n+1})\cap V(F)\neq\emptyset$.
\end{sublemma}

Assume that this fails. By Menger's Theorem and Lemma~\ref{xy-path}, we know that, for each $v \in V(D) - V(F)$, there are at least three internally disjoint $vF$-paths. If two such paths have the same end in $V(F)$, then they form a cycle $C$ such that $V(C)$ meets $V(F)$. It is straightforward to see that~\ref{sublem:absorbtion} holds when $D'=D/C$. Thus, we may now assume that there are three internally disjoint $vF$-paths with distinct endpoints $a$, $b$, and $c$ in $V(F)$. Next, we show the following.

\begin{sublemma}\label{sublem:no path contains two}
    None of the paths $P_1, P_2, P_3, P_4$ contains more than one of $a$, $b$, and $c$.
\end{sublemma}
Suppose that $\{a, b\} \subseteq V(P_i)$ for some $i \in \{1,2,3,4\}$. Without loss of generality, suppose that $\{a, b\} \subseteq V(P_1)$. Observe that if $\{a, b\} = \{x_2, y_1\}$, then $D$ contains a type-A $\Theta$-graph, which contradicts~\ref{assumption}. Hence, we may assume that at least one member of $\{a,b\}$ is an internal vertex of $P_1$. Let $C_v$ be the cycle formed by the $va$-path, the $ab$-subpath of $P_1$, and the $bv$-path (see Figure~\ref{fig:absorbtion} (a)). Then the graphs $D' = D / C_v$ and $F' = F / (E(F) \cap E(C_v))$ satisfy~\ref{sublem:absorbtion}, a contradiction. Thus,~\ref{sublem:no path contains two} holds.\\

 Next we show that at least two of $a, b$, and $c$ are not in $\{x_1, x_2, y_1, y_2\}$. Suppose, without loss of generality, that $\{a,b\}\subseteq \{x_1, x_2, y_1, y_2\}$. By~\ref{sublem:no path contains two}, we know $\{a,b\}\in \{\{x_1,x_2\},\{y_1,y_2\}\}$. Suppose that $\{a,b\}=\{x_1,x_2\}$. Then there is a path in $\{P_1,P_2,P_3,P_4\}$ that contains $c$ and one of $a$ and $b$, contradicting~\ref{sublem:no path contains two}. By symmetry, $\{a,b\}\neq \{y_1,y_2\}$ and hence at least two of $a, b$, and $c$ are not in $\{x_1, x_2, y_1, y_2\}$.

Without loss of generality, we assume that $a \in V(P_1) - \{x_2, y_1\}$. Moreover, by~\ref{sublem:no path contains two}, we know that at least one of $b$ and $c$ belongs to $V(F) - (V(P_1) \cup V(P_3))$. By symmetry, we assume $b \in V(P_2) - \{x_1, y_1\}$. Let $C_v$ be the cycle formed by the $va$-path, the $ay_1$-subpath of $P_1$, the $y_1b$-subpath of $P_2$, and the $bv$-path (see Figure~\ref{fig:absorbtion} (b)). Then the graphs $D' = D / C_v$ and $F' = F / (E(F) \cap E(C_v))$ satisfy~\ref{sublem:absorbtion}, a contradiction. Thus~\ref{sublem:absorbtion} holds, and~\ref{sublem:spanning FPC} follows immediately.

\begin{figure}[htb]
\hbox to \hsize{
\hfil
\resizebox{12cm}{!}{\input{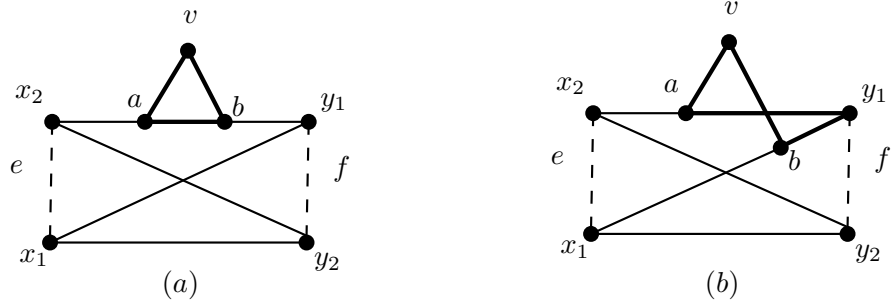}}%
\hfil
}
\caption{Contracting the thickened cycles absorbs $v$ into $F$}\label{fig:absorbtion}
\end{figure}

Next, we prove the following.

\begin{sublemma}\label{sublem:no extra edge}
   Suppose $W$ is a cc-minor of $G$ that contains a spanning four-path connector $R$ and there is a vertex $u\in V(W)-\{x_1,x_2,y_1,y_2\}$. Then $W$ has a cc-minor $W'$ containing a spanning four-path connector $R'$ such that $|V(W) - \{x_1,x_2,y_1,y_2\}| > |V(W') - \{x_1,x_2,y_1,y_2\}|$.
\end{sublemma}

Without loss of generality, suppose $u \in V(P_1) - \{x_2, y_1\}$. By Lemma~\ref{k-edge-connected}, the graph $W$ is 3-edge-connected, and $d_{W}(u) \geq 3$. Therefore, $u$ is incident to an edge $uw \notin E(R)$. Suppose $u$ and $w$ lie on two non-adjacent paths of $R$ (say, $P_1$ and $P_3$ under our assumption). Figure~\ref{fig:uw-path case 3} shows that when $w\notin\{x_1,y_2\}$, when $w=x_1$, and when $w=y_2$, the graph $W$ contains a type-A $\Theta$-graph, a contradiction to~\ref{assumption}.

\begin{figure}[htb]
\hbox to \hsize{
\hfil
\resizebox{12cm}{!}{\input{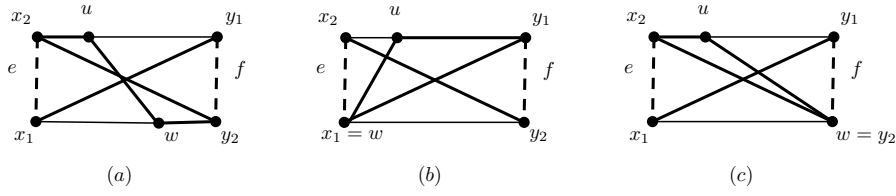}}%
\hfil
}
\caption{The thickened paths, together with the edges $e$ and $f$, form type-A $\Theta$-graphs in each case.}\label{fig:uw-path case 3}
\end{figure}

We may now assume that $w\in V(P_1\cup P_2\cup P_4)-\{x_1,y_2\}$.

(a) Suppose that $w\in V(P_1)$, as shown in Figure~\ref{fig:three cases} (a). Let $C_u$ be the cycle formed by the edge $uw$ and the $wu$-subpath of $P_1$.

(b) Suppose that $w\in V(P_2)-\{x_1,y_1\}$ (or symmetrically $w\in V(P_4)-\{x_2,y_2\}$), as shown in Figure~\ref{fig:three cases} (b). Let $C_u$ be the cycle formed by the edge $uw$, the $wz$-subpath of $P_2$, and the $zu$-subpath of $P_1$. 

In each case, let $W' = W / C_u$ and $R' = R / (E(C_u) \cap E(R))$. Then $R'$ is a spanning four-path connector of $W'$, so~\ref{sublem:no extra edge} holds.
\begin{figure}[htb]
\hbox to \hsize{
\hfil
\resizebox{12cm}{!}{\tikzset{every picture/.style={line width=0.75pt}} 

\begin{tikzpicture}[x=0.75pt,y=0.75pt,yscale=-1,xscale=1]

\draw [color={rgb, 255:red, 0; green, 0; blue, 0 }  ,draw opacity=1 ][line width=1.5]    (341.69,97.56) -- (377.69,115.56) ;
\draw [color={rgb, 255:red, 0; green, 0; blue, 0 }  ,draw opacity=1 ][line width=1.5]    (341.69,97.56) -- (419.69,97.56) ;
\draw [color={rgb, 255:red, 0; green, 0; blue, 0 }  ,draw opacity=1 ][line width=1.5]    (377.69,115.56) -- (419.69,97.56) ;
\draw [color={rgb, 255:red, 0; green, 0; blue, 0 }  ,draw opacity=1 ] [dash pattern={on 4.5pt off 4.5pt}]  (300.69,97.56) -- (299.69,153.56) ;
\draw  [fill={rgb, 255:red, 0; green, 0; blue, 0 }  ,fill opacity=1 ] (296.5,153.56) .. controls (296.5,151.8) and (297.92,150.37) .. (299.69,150.37) .. controls (301.45,150.37) and (302.88,151.8) .. (302.88,153.56) .. controls (302.88,155.33) and (301.45,156.75) .. (299.69,156.75) .. controls (297.92,156.75) and (296.5,155.33) .. (296.5,153.56) -- cycle ;
\draw  [fill={rgb, 255:red, 0; green, 0; blue, 0 }  ,fill opacity=1 ] (297.5,97.56) .. controls (297.5,95.8) and (298.92,94.37) .. (300.69,94.37) .. controls (302.45,94.37) and (303.88,95.8) .. (303.88,97.56) .. controls (303.88,99.33) and (302.45,100.75) .. (300.69,100.75) .. controls (298.92,100.75) and (297.5,99.33) .. (297.5,97.56) -- cycle ;
\draw [color={rgb, 255:red, 0; green, 0; blue, 0 }  ,draw opacity=1 ] [dash pattern={on 4.5pt off 4.5pt}]  (419.69,97.56) -- (418.69,150.37) ;
\draw  [fill={rgb, 255:red, 0; green, 0; blue, 0 }  ,fill opacity=1 ] (415.5,153.56) .. controls (415.5,151.8) and (416.92,150.37) .. (418.69,150.37) .. controls (420.45,150.37) and (421.88,151.8) .. (421.88,153.56) .. controls (421.88,155.33) and (420.45,156.75) .. (418.69,156.75) .. controls (416.92,156.75) and (415.5,155.33) .. (415.5,153.56) -- cycle ;
\draw  [fill={rgb, 255:red, 0; green, 0; blue, 0 }  ,fill opacity=1 ] (416.5,97.56) .. controls (416.5,95.8) and (417.92,94.37) .. (419.69,94.37) .. controls (421.45,94.37) and (422.88,95.8) .. (422.88,97.56) .. controls (422.88,99.33) and (421.45,100.75) .. (419.69,100.75) .. controls (417.92,100.75) and (416.5,99.33) .. (416.5,97.56) -- cycle ;
\draw    (300.69,97.56) -- (341.69,97.56) ;
\draw    (299.69,153.56) -- (418.69,153.56) ;
\draw    (299.69,153.56) -- (377.69,115.56) ;
\draw  [fill={rgb, 255:red, 0; green, 0; blue, 0 }  ,fill opacity=1 ] (338.5,97.56) .. controls (338.5,95.8) and (339.92,94.37) .. (341.69,94.37) .. controls (343.45,94.37) and (344.88,95.8) .. (344.88,97.56) .. controls (344.88,99.33) and (343.45,100.75) .. (341.69,100.75) .. controls (339.92,100.75) and (338.5,99.33) .. (338.5,97.56) -- cycle ;
\draw  [fill={rgb, 255:red, 0; green, 0; blue, 0 }  ,fill opacity=1 ] (374.5,115.56) .. controls (374.5,113.8) and (375.92,112.37) .. (377.69,112.37) .. controls (379.45,112.37) and (380.88,113.8) .. (380.88,115.56) .. controls (380.88,117.33) and (379.45,118.75) .. (377.69,118.75) .. controls (375.92,118.75) and (374.5,117.33) .. (374.5,115.56) -- cycle ;
\draw    (300.69,97.56) -- (418.69,153.56) ;
\draw [color={rgb, 255:red, 0; green, 0; blue, 0 }  ,draw opacity=1 ][line width=1.5]    (150.69,100.56) .. controls (153.2,63.6) and (179.2,59.6) .. (183.69,100.56) ;
\draw  [fill={rgb, 255:red, 0; green, 0; blue, 0 }  ,fill opacity=1 ] (147.5,100.56) .. controls (147.5,98.8) and (148.92,97.37) .. (150.69,97.37) .. controls (152.45,97.37) and (153.88,98.8) .. (153.88,100.56) .. controls (153.88,102.33) and (152.45,103.75) .. (150.69,103.75) .. controls (148.92,103.75) and (147.5,102.33) .. (147.5,100.56) -- cycle ;
\draw [color={rgb, 255:red, 0; green, 0; blue, 0 }  ,draw opacity=1 ][line width=1.5]    (150.69,100.56) -- (183.69,100.56) ;
\draw [color={rgb, 255:red, 0; green, 0; blue, 0 }  ,draw opacity=1 ] [dash pattern={on 4.5pt off 4.5pt}]  (109.69,100.56) -- (108.69,156.56) ;
\draw  [fill={rgb, 255:red, 0; green, 0; blue, 0 }  ,fill opacity=1 ] (105.5,156.56) .. controls (105.5,154.8) and (106.92,153.37) .. (108.69,153.37) .. controls (110.45,153.37) and (111.88,154.8) .. (111.88,156.56) .. controls (111.88,158.33) and (110.45,159.75) .. (108.69,159.75) .. controls (106.92,159.75) and (105.5,158.33) .. (105.5,156.56) -- cycle ;
\draw  [fill={rgb, 255:red, 0; green, 0; blue, 0 }  ,fill opacity=1 ] (106.5,100.56) .. controls (106.5,98.8) and (107.92,97.37) .. (109.69,97.37) .. controls (111.45,97.37) and (112.88,98.8) .. (112.88,100.56) .. controls (112.88,102.33) and (111.45,103.75) .. (109.69,103.75) .. controls (107.92,103.75) and (106.5,102.33) .. (106.5,100.56) -- cycle ;
\draw [color={rgb, 255:red, 0; green, 0; blue, 0 }  ,draw opacity=1 ] [dash pattern={on 4.5pt off 4.5pt}]  (228.69,100.56) -- (227.69,153.37) ;
\draw  [fill={rgb, 255:red, 0; green, 0; blue, 0 }  ,fill opacity=1 ] (224.5,156.56) .. controls (224.5,154.8) and (225.92,153.37) .. (227.69,153.37) .. controls (229.45,153.37) and (230.88,154.8) .. (230.88,156.56) .. controls (230.88,158.33) and (229.45,159.75) .. (227.69,159.75) .. controls (225.92,159.75) and (224.5,158.33) .. (224.5,156.56) -- cycle ;
\draw  [fill={rgb, 255:red, 0; green, 0; blue, 0 }  ,fill opacity=1 ] (225.5,100.56) .. controls (225.5,98.8) and (226.92,97.37) .. (228.69,97.37) .. controls (230.45,97.37) and (231.88,98.8) .. (231.88,100.56) .. controls (231.88,102.33) and (230.45,103.75) .. (228.69,103.75) .. controls (226.92,103.75) and (225.5,102.33) .. (225.5,100.56) -- cycle ;
\draw    (109.69,100.56) -- (150.69,100.56) ;
\draw    (108.69,156.56) -- (227.69,156.56) ;
\draw    (108.69,156.56) -- (228.69,100.56) ;
\draw  [fill={rgb, 255:red, 0; green, 0; blue, 0 }  ,fill opacity=1 ] (180.5,100.56) .. controls (180.5,98.8) and (181.92,97.37) .. (183.69,97.37) .. controls (185.45,97.37) and (186.88,98.8) .. (186.88,100.56) .. controls (186.88,102.33) and (185.45,103.75) .. (183.69,103.75) .. controls (181.92,103.75) and (180.5,102.33) .. (180.5,100.56) -- cycle ;
\draw    (109.69,100.56) -- (227.69,156.56) ;
\draw    (183.69,100.56) -- (227.88,100.56) ;

\draw (159,168.9) node [anchor=north west][inner sep=0.75pt]    {$( a)$};
\draw (350,168.9) node [anchor=north west][inner sep=0.75pt]    {$( b)$};
\draw (88.5,116.86) node [anchor=north west][inner sep=0.75pt]  [color={rgb, 255:red, 0; green, 0; blue, 0 }  ,opacity=1 ]  {$e$};
\draw (145,105.4) node [anchor=north west][inner sep=0.75pt]    {$u$};
\draw (177,105.4) node [anchor=north west][inner sep=0.75pt]    {$w$};
\draw (238.5,114.86) node [anchor=north west][inner sep=0.75pt]  [color={rgb, 255:red, 0; green, 0; blue, 0 }  ,opacity=1 ]  {$f$};
\draw (279.5,113.86) node [anchor=north west][inner sep=0.75pt]  [color={rgb, 255:red, 0; green, 0; blue, 0 }  ,opacity=1 ]  {$e$};
\draw (382.88,118.96) node [anchor=north west][inner sep=0.75pt]    {$w$};
\draw (338,79.4) node [anchor=north west][inner sep=0.75pt]    {$u$};
\draw (425,81.4) node [anchor=north west][inner sep=0.75pt]    {$z$};
\draw (429.5,111.86) node [anchor=north west][inner sep=0.75pt]  [color={rgb, 255:red, 0; green, 0; blue, 0 }  ,opacity=1 ]  {$f$};

\end{tikzpicture}}%
\hfil
}
\caption{Two cases of a $uw$-edge}\label{fig:three cases}
\end{figure}
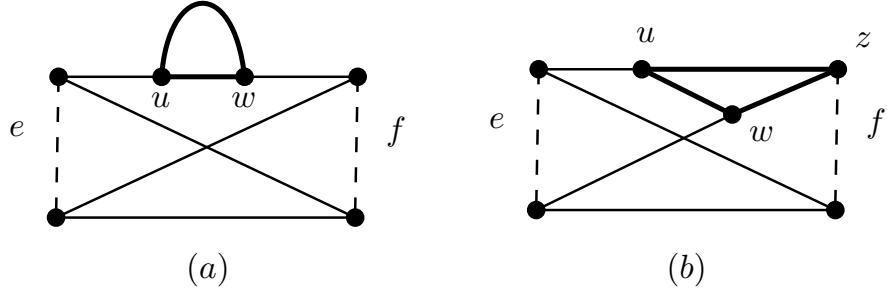

Applying~\ref{sublem:no extra edge} inductively on $G_s'$ and $K'$, we conclude that $G_s'$ has a cc-minor $H$ containing a four-path connector. Moreover, $V(H) = \{x_1, x_2, y_1, y_2\}$. It is not difficult to see that the edges $e$ and $f$ are non-adjacent in $H$ and that the simplification of $H$ is isomorphic to $K_4$. Suppose that there is a pair of parallel edges $g$ and $h$ in $E(H)$ that are not parallel to $e$ or $f$. Because the simplification of $H$ is isomorphic to $K_4$, there is a 4-cycle $U$ that contains $e,f$ and $g$. However, the graph with edge set $U\cup h$ is a type-A $\Theta$-graph, a contradiction to~\ref{assumption}. Hence, Theorem~\ref{3-con cc_minor} holds.
\end{proof}
\section{cc-minors of large $3$-connected graphs}\label{large 3-con}

In this section, we determine a set of unavoidable cc-minors of sufficiently large 3-connected graphs.

\begin{theorem}\label{large_3_con}
    For every integer $t\geq 3$, there is a function $f_{\ref{large_3_con}}(t)$ such that if a simple $3$-connected graph $G$ has more than $f_{\ref{large_3_con}}(t)$ edges, then $G$ has a fan-type graph $F_{t_1, t_2, \dots, t_n}$ as a cc-minor such that $\sum_{i=1}^n t_i \geq t$.
\end{theorem}

Before beginning the proof of Theorem~\ref{large_3_con}, we present three lemmas. The first is a Ramsey-type result for $3$-connected graphs; the second is a result for weighted trees, where we use the latter as auxiliary graphs in our analysis.

Let $k$ be an integer exceeding two. Figure~\ref{fig:W_k,L_k,V_k} shows three families of graphs that we now describe. The \textit{$k$-rung ladder} $L_k$ has vertices $v_1, v_2, \dots, v_k,u_1, u_2, \dots, u_k$, where $v_1, v_2, \dots, v_k$ and $u_1, u_2, \dots, u_k$ form paths in the listed order, and $v_i$ is adjacent to $u_i$ for each $i \in\{ 1, 2, \dots, k\}$. The graph $V_k$ is obtained from $L_k$ by adding an edge between $v_1$ and $v_k$ and contracting the edges joining $u_1$ to $v_1$ and $u_k$ to $v_k$. The $k$-spoke wheel is denoted by $W_k$. Oporowski, Oxley, and Thomas~\cite{OOT1993} characterized a set of unavoidable structures of large $3$-connected graphs as follows.

\begin{figure}[htb]
\hbox to \hsize{
\hfil
\resizebox{12cm}{!}{\input{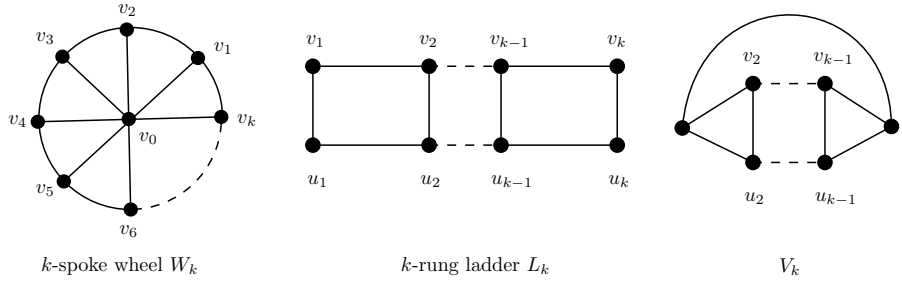}}%
\hfil
}
\caption{Some important graphs for Lemma~\ref{3-con OOT}}
\label{fig:W_k,L_k,V_k}
\end{figure}
\begin{lemma}\label{3-con OOT}
   For every integer $k\geq 3$, there is a function $f_{\ref{3-con OOT}}(k)$ such that every $3$-connected graph with at least $f_{\ref{3-con OOT}}(k)$ vertices contains a subgraph isomorphic to a subdivision of one of $W_k$, $V_k$, and $K_{3,k}$.
\end{lemma}

A \textit{weighted tree} is a tree $T$ together with a weight function $w$ such that each vertex $v$ is assigned a non-negative integer-valued weight $w(v)$, and $w(A) = \sum_{v \in A} w(v)$ for each $A\subseteq V(T)$.

In the next lemma, we use the notion of the {\it center} of a graph. This is the set of vertices with the smallest maximum distance to other vertices. It is well known that, in a tree, the center consists of either a single vertex or two adjacent vertices of the tree.

\begin{lemma}
\label{weighted_tree}
    For every integer $t > 1$, there is a function $f_{\ref{weighted_tree}}(t)$ such that every weighted tree $T$ with $w(V(T)) > f_{\ref{weighted_tree}}(t)$ contains one of the following:
    \begin{itemize}
        \item[(\romannum{1})] a vertex $v$ such that $d(v) > t$;
        \item[(\romannum{2})] a path $P$ such that $|V(P)| > t$; or
        \item[(\romannum{3})] a path $P$ such that $w(V(P)) > t$.
    \end{itemize}
\end{lemma}
\begin{proof}
    We prove that $f_{\ref{weighted_tree}}(t) = \sum_{i=1}^{\lfloor \frac{t}{2} \rfloor+1} t^i$ satisfies the condition. Let $T$ be a tree. We may assume that both the maximum degree and the number of vertices in a longest path do not exceed $t$ otherwise (\romannum{1}) or (\romannum{2}) holds. By grouping the vertices of $T$ based on their distance from a center vertex of $T$, we see that $|V(T)| \leq \sum_{j=0}^{\lfloor \frac{t}{2} \rfloor} t^j$. However, since $w(T) \geq 1+\sum_{i=1}^{\lfloor \frac{t}{2} \rfloor+1} t^i$, there must be a vertex $v$ such that $w(v) > t$. Therefore, each path $P$ containing $v$ satisfies $w(V(P)) > t$, so (\romannum{3}) holds.
\end{proof}

We are now ready to prove the main result of this section.\\

\noindent{\it Proof of Theorem~\ref{large_3_con}.}
Let $H$ be a subgraph of $G$. An edge $e$ of $E(G)- E(H)$ is an \textit{$H$-bridge} if $e$ is incident with at least one vertex of $H$. We call each connected component of $G - V(H)$ an \textit{$H$-island}. Note that an $H$-bridge is either
\begin{itemize}
    \item[(\romannum{1})] an edge having both vertices in $V(H)$, or
    \item[(\romannum{2})] an edge having one vertex in $V(H)$ and one vertex in an $H$-island.
\end{itemize}

We first prove the following.
\begin{sublemma}\label{cycle to fan}
    If $G$ has a cycle $C$ and a $C$-island $I$ such that there are at least $f_{\ref{weighted_tree}}(t)$ $C$-bridges between $C$ and $I$, then $G$ has a fan-type graph $F_{t_1, t_2, \dots, t_n}$ as a cc-minor where $\sum_{i=1}^n t_i \geq t$.
\end{sublemma}

First, observe that $G - V(I)$ does not have a cut edge. Thus we can contract all of the edges of $G-V(I)$ by successively contracting a sequence of cycles. The resulting graph $G'$ is obtained from $G[V(C) \cup V(I)]$ by contracting $C$. We denote by $c$ the vertex that results by identifying all of the vertices of $C$. Note that $d_{G'}(c) \geq f_{\ref{weighted_tree}}(t)$ and the neighbors of $c$ in $G'$ are contained in $V(I)$. Let $G_0 = G'$. If $G_i - c$ has a cycle $C_i$, define $G_{i+1} = G_i / C_i$. This process results in a sequence $G_0, G_1, \dots, G_s$ such that $G_s - c$ is a tree $T$.

Now, define a weight function $w$ on $V(T)$ by, for each vertex $v$ of $T$, letting $w(v)$ be the number of edges joining $c$ and $v$. Clearly, $w(V(T)) = d_{G'}(c) \geq f_{\ref{weighted_tree}}(t)$. By Lemma~\ref{weighted_tree}, $T$ has a subgraph $T'$ that is one of the following:
\begin{enumerate}[label=(\roman*), itemsep=0pt, topsep=0pt]
    \item a vertex $v$ such that $d(v) > t$;
    \item a path $P$ such that $|V(P)| > t$; or
    \item a path $P$ such that $w(V(P)) > t$.
\end{enumerate}

Let $T_0=T$. Assume that we have defined a sequence $(G_s,T_0),(G_{s+1},T_1),\dots,\\(G_{s+i},T_i)$ where each $G_{s+j}$ is $3$-edge-connected having the tree $T_j$ as a subgraph. If $T_i \neq T'$, then $T_i$ has a leaf $l \notin V(T')$. Since $G_{s+i}$ is $3$-edge-connected, there are two edges joining $c$ and $l$ that form a cycle $O_i$ in $G_{s+i}$. Define $G_{s+i+1} = G_{s+i}/ O_i$ and $T_{i+1} = T-l$. Repeating this process, we eventually obtain a pair $(G_{s+h},T_h)$ with $T_h=T'$. By the choice of $T'$, we see that if $T'$ is a vertex of degree more than $t$ in $G_{s+h}$, then $G_{s+h}$ is a bond graph with more than $t$ edges; if $T'$ is a path on more than $t$ vertices, then $G_{s+h}$ is a fan-type graph with more than $t$ sets of parallel spokes; and if $T'$ is a path $P$ such that $w(V(P)) > t$, then $G_{s+h}$ is a fan-type graph with more than $t$ spokes. Therefore,~\ref{cycle to fan} holds.\\

To complete the proof of the theorem, we shall show that the required result holds for the function $f_{\ref{large_3_con}}(t) = {f_{\ref{3-con OOT}} \circ f_{\ref{weighted_tree}}(t) \choose 2}$. Since $G$ is simple and has more than ${f_{\ref{3-con OOT}} \circ f_{\ref{weighted_tree}}(t) \choose 2}$ edges, $G$ has more than $f_{\ref{3-con OOT}} \circ f_{\ref{weighted_tree}}(t)$ vertices. By Lemma~\ref{3-con OOT}, $G$ has a subgraph isomorphic to a subdivision of one of $W_k$, $V_k$, or $K_{3, k}$ where $k=f_{\ref{weighted_tree}}(t)$. In each of these three cases, let $C$ be the bold cycle and $I$ be the $C$-island containing the white vertices, as shown in Figure~\ref{fig:island in Wk Vk K3k}. It is straightforward to verify that the choices of $C$ and $I$ satisfy the conditions in~\ref{cycle to fan}. Hence, by~\ref{cycle to fan}, $G$ has a fan-type graph $F_{t_1, t_2, \dots, t_n}$ as a cc-minor such that $\sum_{i=1}^n t_i \geq t$. \qed

\begin{figure}[htb]
\hbox to \hsize{
\hfil
\resizebox{12cm}{!}{\input{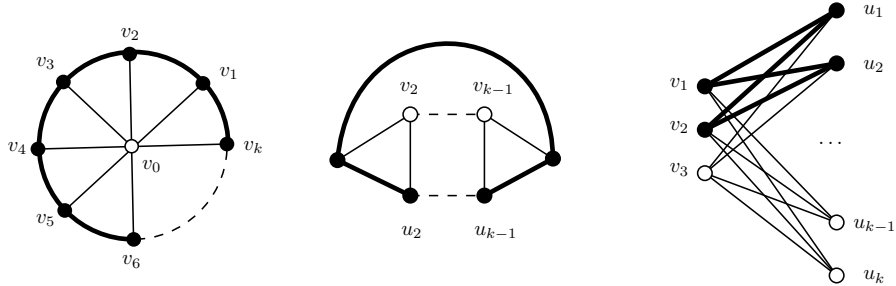}}%
\hfil
}
\caption{$C$ and a $C$-island $I$ in each of $W_k$, $V_k$, and $K_{3,k}$}\label{fig:island in Wk Vk K3k}
\end{figure}

\section{Proof of Theorem~\ref{main}}\label{proof main}

Before we present the proof of Theorem~\ref{main}, we prove two lemmas.

\begin{lemma}\label{lem:cc-minor of parallel-ext}
    If $H$ is a cc-minor of a loopless graph $G$ and $G'$ is a parallel-path extension of $G$, then $G'$ has a cc-minor $H'$ that is a parallel-path extension of $H$.
\end{lemma}

\begin{proof}
    By Proposition~\ref{prop:structure}, there is a collection $\{G_1,G_2,\dots,G_k\}$ of disjoint $2$-edge-connected subgraphs of $G$ such that $H=G/(\cup_{i=1}^kE(G_k))$. Now $G'$ is obtained from $G$ by replacing some edges with internally disjoint paths joining their ends. For each $i\in\{1,2,\dots,k\}$, let $G_i'$ be the graph that is obtained from $G_i$ by replacing all such edges in $G_i$ with the same set of internally disjoint paths joining their ends as in $G'$. Clearly, $G_i'$ is $2$-edge-connected for all $i\in\{1,2,\dots,k\}$. Let $H'=G'/(\cup_{i=1}^kE(G_i'))$. It is straightforward to check that $H'$ is a parallel-path extension of $H$.
\end{proof}

\begin{lemma}\label{lem:2-sum of cc-minor}
    Let $G$ be the $2$-sum of two loopless graphs $G_1$ and $G_2$ on the basepoint $b$. Suppose that, for each $i \in \{1,2\}$, $H_i$ is a cc-minor of $G_i$ that has $b$ as a non-loop edge. If $H$ is the $2$-sum of $H_1$ and $H_2$ on the basepoint $b$, then $H$ is a cc-minor of $G$.
\end{lemma}

\begin{proof}
    By Proposition~\ref{prop:structure}, for each $i$ in $\{1,2\}$, there is a collection $\mathcal{J}_i$ of disjoint $2$-edge-connected subgraphs of $G_i$ such that $H_i=G_i/(\cup_{J\in\mathcal{J}_i}E(J))$. Because $E(H_1)\cap E(H_2)=\{b\}$ and $b\notin E(J)$ for each $J$ in $\mathcal{J}_1\cup\mathcal{J}_2$, we know $\mathcal{J}_1\cup\mathcal{J}_2$ is a collection of edge-disjoint $2$-edge-connected subgraphs of $G$. By Corollary~\ref{cor:contracting subgraph}, we deduce that $H$, which equals $G/(\cup_{J\in\mathcal{J}_1\cup\mathcal{J}_2}E(J))$, is a cc-minor of $G$.
\end{proof}

We are now ready to prove the main theorem of the paper.\\

\noindent{\it Proof of Theorem~\ref{main}.} Note that proving Theorem~\ref{main} is equivalent to proving the following.
\setcounter{theorem}{3}
\begin{sublemma}\label{sublem:main(template)}
    Let $r$ be a positive integer. There is an integer $g(r)$ such that every loopless $2$-connected graph $G$ with $|E(G)| \geq g(r)$ has a cc-minor $H$ that is a parallel-path extension of a template with at least $r$ parts.
\end{sublemma}

We shall show that $G$ has a parallel-path extension of a template with at least $r$ parts as a cc-minor when $g(r) =\sum_{i=1}^{r} \big(f_{\ref{large_3_con}}(r+2)\big)^i$ where $ f_{\ref{large_3_con}}$ is the function whose existence was established in Theorem~\ref{large_3_con}. Let $T$ be a tree decomposition of $G$ such that each vertex of $T$ is either a simple $3$-connected graph, or $K_3$, or $B_3$, as described in Theorem~\ref{thm:tree-decomp}. First, we show the following.

\begin{sublemma}\label{sub:no large block}
    If there is a vertex $G_v \in V(T)$ such that $|E(G_v)| > f_{\ref{large_3_con}}(r+2)$, then~\ref{sublem:main(template)} holds.
\end{sublemma}

Note that $|f_{\ref{large_3_con}}(t)| \geq t$ for all $t\geq 3$, so, for any positive integer $r$, we have $|E(G_v)| > f_{\ref{large_3_con}}(r+2) \geq 3$. Thus, $G_v$ cannot be isomorphic to $K_3$ or $B_3$. Therefore, $G_v$ must be a simple 3-connected graph. Note that each component of $T - G_v$ is a tree decomposition for a $2$-connected graph. Let $\{J_1, J_2, \dots, J_n\}$ be the collection of such graphs. Then $G$ can be obtained by repeatedly gluing each $J_i$ to $G_v$ via a $2$-sum on the basepoint $b_i$ where $\{b_i\}=E(J_i) \cap E(G_v)$ for all $i \in \{1,2,\dots,n\}$. By Lemma~\ref{2con_to_parallel}, for each $i \in \{1,2,\dots,n\}$, the graph $J_i$ has a cc-minor $J_i'$ that is a parallel connection, with basepoint $b_i$, of a collection of cycles containing $b_i$. Thus, $G$ has a cc-minor $G_v'$ that is obtained by, for each $i$ in $\{1,2\dots,n\}$, gluing $J_i'$ to $G_v$ via a $2$-sum on the basepoint $b_i$. It is straightforward to verify that $G_v'$ is a parallel-path extension of $G_v$.

By Lemma~\ref{lem:cc-minor of parallel-ext}, it suffices to show that $G_v$ has a cc-minor that is a template with at least $r$ parts. By Theorem~\ref{large_3_con}, $G_v$ has a cc-minor that is a fan-type graph with at least $r+2$ spokes, which constitutes a template with at least $r$ parts. Hence~\ref{sub:no large block} holds.\\

We may now assume that $|E(G_v)| \leq f_{\ref{large_3_con}}(r+2)$ for every vertex $G_v$ in $V(T)$. First we show the following.

\begin{sublemma}\label{sublem:long_path}
    $T$ has a path with at least $r$ vertices.
\end{sublemma}

Since the basepoints are deleted after $2$-sums, for each $G_v\in V(T)$, we have $d_T(G_v)\leq |E(G_v)|$. Therefore, we conclude that $d_T(G_v) \leq f_{\ref{large_3_con}}(r+2)$ for every $G_v\in V(T)$. For any two vertices $G_u$ and $G_v$, the {\it distance} $d(G_u,G_v)$ between them is the number of edges of the shortest $G_uG_v$-path in $T$. For an arbitrary vertex $G_w$ in $V(T)$, we have the following.

\begin{sublemma}\label{sublem:big height}
For each non-negative integer $h$, 
\[|\{G_v\in V(T):d(G_w,G_v)=h\}|\leq \big(f_{\ref{large_3_con}}(r+2)\big)^h.\]
\end{sublemma}

Since $|E(G)|\geq g(r)$ and $|E(G_v)| \leq f_{\ref{large_3_con}}(r+2)$ for each vertex $G_v$ in $V(T)$, the tree $T$ has at least $\frac{g(r)}{f_{\ref{large_3_con}}(r+2)}$ vertices. As $g(r)=\sum_{i=1}^{r}\big( f_{\ref{large_3_con}}(r+2)\big)^i$, we deduce that $T$ has at least $\sum_{j=0}^{r-1} \big(f_{\ref{large_3_con}}(r+2)\big)^j$ vertices. Therefore, by~\ref{sublem:big height}, there is a vertex $G_q\in V(T)$ such that $d(G_w,G_q)\geq r-1$. Hence~\ref{sublem:long_path} holds.\\

Let $P$ be a path $G_1 G_2 \dots G_r$ in $T$. Note that $P$ is a graph-labeled tree representing a graph $G_P$ obtained from the graphs $G_1, G_2, \dots, G_r$ by applying a sequence of $2$-sums. Note that each component of $T-V(P)$ is a tree decomposition for a $2$-connected graph. Let $\{F_1,F_2,\dots,F_m\}$ be the collection of these graphs. Then $G$ can be obtained by, for each $i$ in $\{1,2,\dots,m\}$, gluing $F_i$ to $G_P$ via a $2$-sum on the basepoint $p_i$ where $\{p_i\}=E(F_i) \cap E(G_P)$. By Lemma~\ref{2con_to_parallel}, for each $i \in \{1,2,\dots,m\}$, the graph $F_i$ has a cc-minor $F_i'$ that is a parallel connection, with basepoint $p_i$, of a collection of cycles containing $p_i$. Thus, $G$ has a cc-minor $G_P'$ that is obtained by, for each $i$ in $\{1,2,\dots,m\}$, gluing $F_i'$ to $G_P$ via a $2$-sum on the basepoint $p_i$. It is straightforward to verify that $G_P'$ is a parallel-path extension of $G_P$. By Lemma~\ref{lem:cc-minor of parallel-ext}, it remains only to show that $G_P$ has a cc-minor that is a template with at least $r$ parts.

For each $i\in\{1,2,\dots,r-1\}$, let $e_i$ be the unique edge in $E(G_i)\cap E(G_{i+1})$ that is used as the basepoint between $G_i$ and $G_{i+1}$ in the construction of $G_P$. For each $i\in\{2,3,\dots,r-1\}$, by Theorem~\ref{3-con cc_minor}, $G_i$ has a cc-minor $G_i'$ that is
\begin{enumerate}
    \item[(\romannum{1})] a bond graph containing $e_{i-1}$ and $e_{i+1}$; or
    \item[(\romannum{2})] a fan-type graph containing $e_{i-1}$ and $e_{i+1}$ as distinct outer spokes that are not parallel; or
    \item[(\romannum{3})] a parallel extension of $K_4$ that has $e_{i-1}$ and $e_{i+1}$ as non-adjacent edges.
\end{enumerate}
For $i \in \{1, r\}$, if $G_i$ is isomorphic to $K_3$ or $B_3$, let $G_i' = G_i$. Now suppose that $G_1$ is not isomorphic to $K_3$ or $B_3$. Then, by Corollary~\ref{cor:3-edge-con to parallel}, $G_1$ has a cc-minor $G_1'$ that contains $e_1$ and is isomorphic to a bond graph with at least three edges. Define $G_r'$ symmetrically when $G_r$ is not isomorphic to $K_3$ or $B_3$. Let $G_P'$ be the graph that is obtained by applying a sequence of $2$-sums to the graphs $G_1', G_2', \dots, G_r'$, using the edges $e_1, e_2, \dots, e_{r-1}$ as basepoints. Applying Lemma~\ref{lem:2-sum of cc-minor} inductively, we see that $G_P'$ is a cc-minor of $G_P$.
It is straightforward to verify that $G_P'$ is a template on at least $r$ parts.
\qed

\section*{Acknowledgments}
The authors thank the anonymous referees for their helpful comments and suggestions, which improved the presentation of this paper.

\end{document}